\documentclass{amsart}
\pdfoutput=1

\usepackage{amsmath, mathtools}
\usepackage{amsfonts}
\usepackage{amssymb}
\usepackage{amsthm}
\usepackage{tikz-cd}
\usepackage{tikz}
\usepackage{amsxtra}
\usepackage{wrapfig}
\usepackage{color}
\usepackage{url}
\usepackage[hidelinks]{hyperref}
\usepackage[percent]{overpic}
\usepackage{mathrsfs}
\usepackage{stmaryrd}

\usepackage{setspace}

\newtheorem{thm}{Theorem}[subsection]
\newtheorem{prop}[thm]{Proposition}

\newtheorem{cor}[thm]{Corollary}
\newtheorem*{coro}{Corollary}

\newtheorem{lem}[thm]{Lemma}
\theoremstyle{definition}
\newtheorem{defn}[thm]{Definition}
\newtheorem{setup}[thm]{Setup}
\newtheorem*{mainthm}{Main Theorem}
\newtheorem*{thrm}{Theorem}
\newtheorem*{conji}{Conjecture}
\newtheorem{conj}[thm]{Conjecture}
\theoremstyle{remark}
\newtheorem{ex}[thm]{Example}
\newtheorem{rmk}[thm]{Remark}

\usepackage{textcomp}
\DeclareSymbolFont{ugrf@m}{U}{eur}{m}{n}
\DeclareMathSymbol{\upmu}{\mathord}{ugrf@m}{"16}

\newcommand{\cat}[1]{{\mathbf{#1}}}
\newcommand{\p}{}
\newcommand{\spec}{\operatorname{Spec}}
\newcommand{\onto}{\twoheadrightarrow}
\newcommand{\into}{\hookrightarrow}
\newcommand{\from}{\leftarrow}
\newcommand{\lot}{\otimes^{\mathbb{L}}}
\newcommand{\per}{{\ensuremath{\cat{per}}}\kern 1pt}
\newcommand{\stab}{\underline{\mathrm{CM}}}

\DeclareMathOperator{\id}{id}
\let\im\relax\DeclareMathOperator{\im}{im}
\let\ker\relax\DeclareMathOperator{\ker}{ker}
\DeclareMathOperator{\coker}{coker}

\let\hom\relax\newcommand{\hom}{\mathrm{Hom}}
\newcommand{\enn}{\mathrm{End}}
\DeclareMathOperator{\tor}{Tor}
\DeclareMathOperator{\ext}{Ext}
\newcommand{\cell}{\operatorname{Cell}}

\newcommand{\Z}{\mathbb{Z}}
\newcommand{\N}{\mathbb{N}}
\newcommand{\A}{\mathbb{A}}

\newcommand{\R}{{\mathrm{\normalfont\mathbb{R}}}}

\numberwithin{equation}{section}

\newcommand{\con}{\mathrm{con}}
\newcommand{\dca}{{\ensuremath{A^\mathrm{der}_\con}}}
\newcommand{\dq}{\ensuremath{A/^{\mathbb{L}}\kern -2pt AeA} }
\newcommand{\dqb}{\ensuremath{B/^{\mathbb{L}}\kern -2pt BeB} }
\newcommand{\thick}{\ensuremath{\cat{thick} \kern 0.5pt}}

\newcommand{\dgh}{\mathrm{HOM}}
\newcommand{\dge}{\mathrm{END}}

\newcommand{\recol}{\mathrel{\substack{\textstyle\leftarrow\\[-0.6ex]
			\textstyle\rightarrow \\[-0.6ex]
			\textstyle\leftarrow}}}

\newcommand{\dloc}{\mathbb{L}_S(A)}

\newcommand{\mc}{\mathscr{M}\kern -0.7pt C}

\newcommand{\ggr}{\mathscr{G}\kern -1pt g}

\makeatletter
\newcommand{\holim@}[2]{%
	\vtop{\m@th\ialign{##\cr
			\hfil$#1\operator@font holim$\hfil\cr
			\noalign{\nointerlineskip\kern1.5\ex@}#2\cr
			\noalign{\nointerlineskip\kern-\ex@}\cr}}%
}
\newcommand{\holim}{%
	\mathop{\mathpalette\holim@{\leftarrowfill@\textstyle}}\nmlimits@
}
\makeatother

\usepackage[utf8]{inputenc}
\usepackage[T1]{fontenc}
\usepackage{textcomp}

\allowdisplaybreaks

\calclayout

\begin{document}

		\title{Singularity categories via the derived quotient}
	\author{Matt Booth}
	
	\address{Universiteit Antwerpen,
		Departement Wiskunde-Informatica,
Campus Middelheim,
		Middelheimlaan 1,
		2020 Antwerpen,
		Belgium}
	
	\email{matt.booth@uantwerpen.be}
	
	\urladdr{mattbooth.info}



	\begin{abstract}
Given a noncommutative partial resolution $A=\enn_R(R\oplus M)$ of a Gorenstein singularity $R$, we show that the relative singularity category $\Delta_R(A)$ of Kalck--Yang is controlled by a certain connective dga $\dq$, the derived quotient of Braun--Chuang--Lazarev. We think of $\dq$ as a kind of `derived exceptional locus' of the partial resolution $A$, as we show that it can be thought of as the universal dga fitting into a suitable recollement. This theoretical result has geometric consequences. When $R$ is an isolated hypersurface singularity, it follows that the singularity category $D_\mathrm{sg}(R)$ is determined completely by $\dq$, even when $A$ has infinite global dimension. Thus our derived contraction algebra classifies threefold flops, even those $X \to  \spec (R)$ where $X$ has only terminal singularities. This gives a solution to the strongest form of the derived Donovan--Wemyss conjecture, which we further show is the best possible classification result in this singular setting.

\end{abstract}
	
	\maketitle

	\section{Introduction} 

Recently, the study of noncommutative rings with idempotents, and in particular their homological properties, has become important within algebraic geometry \cite{bikr, kiwy, DWncdf}. Their study in this context was initiated by Van den Bergh \cite{vdbnccr, vdb}, who introduced noncommutative crepant resolutions (NCCRs) as a noncommutative notion of a resolution of singularities amenable to techniques from birational geometry.

One such use of this technology is the conjectural Donovan--Wemyss classification of smooth simple threefold flops (over an algebraically closed field $k$ of characteristic zero). Given such a flopping contraction $X \to \spec (R)$, one can construct an NCCR $A$ of $R$ together with a derived equivalence $D^b(A)\simeq D^b(X)$ \cite{vdb}. The quotient of $A$ by the idempotent $e$ is known as the contraction algebra; it is a finite-dimensional local $k$-algebra. It is conjectured that this algebra classifies smooth simple threefold flops complete locally.
\begin{conji}[Donovan--Wemyss {\cite[1.4]{DWncdf}}]
	Let $X \to \spec R$ and $X' \to \spec R'$ be flopping contractions of an irreducible rational curve in a smooth projective threefold, where $R$ and $R'$ are complete local rings. If the associated contraction algebras are isomorphic, then $R \cong  R'$.
\end{conji}
In the singular setting, although one can still define the contraction algebra the conjecture is false (\ref{singularnope}). Hence, if one wishes to classify singular flops, more information is necessary. To rectify this failure, we move to the derived setting and replace the contraction algebra $A/AeA$ with the derived quotient $\dq$.
 
\p Let $A$ be a $k$-algebra with an idempotent $e$. One can take the derived quotient $\dq$ of $A$ by $e$, which by Braun, Chuang, and Lazarev is the differential graded algebra (dga) universal\footnote{The universal property only defines $\dq$ up to quasi-isomorphism of $A$-algebras.} with respect to homotopy annihilating $e$ \cite{bcl}. In fact, below we show that it also has a universal property with respect to recollements of derived categories generated by idempotents (\ref{recoll}, \ref{derivedmoritarmk}), which gives a functorial description of a recollement of Kalck and Yang \cite{kalckyang}. This formalism allows us to easily recover many statements about recollements already found in the literature (\ref{recolrmk1}, \ref{recolrmk2}).

The objective of this paper is to use these derived quotients $\dq$ in the setting of noncommutative partial resolutions. These generalise Van den Bergh's NCCRs and appear for example in the homological MMP as rings derived equivalent to partial crepant resolutions \cite{hmmp}. These partial resolutions are rings of the form $A=\enn_R(R\oplus M)$, and for such a partial resolution $e$ denotes the idempotent $\id_R \in A$. The main advantage for us of working in such generality is that we can entirely remove smoothness hypotheses on $A$, which are strictly necessary in \cite{DWncdf} and \cite{huakeller}. Our main result is the following.

\begin{mainthm}[\ref{recov}]
	Let $R$ be a complete local isolated hypersurface singularity. Let $M$ be an indecomposable non-projective maximal Cohen--Macaulay (MCM) $R$-module and let $A\coloneqq \enn_{R}(R\oplus M)$ be the associated noncommutative partial resolution of $R$. Then the dga quasi-isomorphism class of the derived exceptional locus\footnote{For an explanation of why we use this terminology, see \ref{delrmk}.} $\dq$ recovers $R$ amongst all complete local isolated hypersurface singularities of the same (Krull) dimension as $R$.
\end{mainthm}
In particular, $A$ is allowed to have infinite global dimension. This leads to the following corollary, which we show in Example \ref{singularnope} is the strongest possible classification result for singular threefold flops.
\begin{coro}[singular derived Donovan-Wemyss conjecture; see \ref{ddwc}]
	Let $X \to \spec R$ and $X' \to \spec R'$ be flopping contractions of an irreducible rational curve in a projective threefold with only terminal singularities, where $R$ and $R'$ are complete local rings with an isolated singularity. If the derived contraction algebras are quasi-isomorphic, then $R \cong  R'$.
	\end{coro}
If $X\to\spec(R)$ is a smooth flopping contraction then $R$ necessarily has isolated singularities, and it follows that the Donovan--Wemyss conjecture is true if one substitutes the derived contraction algebra for the usual contraction algebra. 

\p The proof of our main result goes via the (dg) singularity category $D_\mathrm{sg}(R)$ of the base $R$. Indeed, a recent theorem of Hua and Keller \cite{huakeller} states that the quasi-equivalence class of $D_\mathrm{sg}(R)$, together with the Krull dimension, is enough to recover $R$, so we reduce to showing that the derived exceptional locus recovers the singularity category. As a first step in this direction, we prove the following key technical theorem, which may be of independent interest.

\begin{thrm}
Let $A=\enn_R(R\oplus M)$ be any noncommutative partial resolution of a complete local Gorenstein singularity $R$.
	\begin{enumerate}
		\item (\ref{qisolem}) The derived exceptional locus $\dq$ is the truncation to nonpositive degrees of the derived endomorphism dga of $M\in D_\mathrm{sg}(R)$.
		\item (\ref{etaex}) When $R$ is in addition a hypersurface, the derived endomorphism dga of $M$ is in fact the derived localisation \cite{bcl} of $\dq$ at a suitably unique degree -2 ‘periodicity element’.
	\end{enumerate}
	Thus, when $R$ is a hypersurface, it is possible to recover the derived endomorphism dga of $M$ from the dga $\dq$ (\ref{uniqueetaqi}).
\end{thrm}

\p Key to our arguments will be the use of Buchweitz's stable category of $R$ \cite{buchweitz}, which is why we need to assume that $M$ is MCM; note that $M$ vanishes in the stable category if and only if it is projective. The assumption that M is MCM is automatic in our main application to flops (see \cite{hmmp}). We remark that part (1) of the previous theorem is true in slightly more generality (see \ref{premcmrmk}).

\p This paper is an adaptation and improvement of the second part of the author's PhD thesis \cite{me}, the main idea of which is to regard $\dq$ as a derived version of the Donovan--Wemyss contraction algebra \cite{DWncdf, hmmp, enhancements, contsdefs}. Here we focus on the more algebraic aspects of $\dq$ related to the Donovan--Wemyss conjecture; in other papers we will explore the derived deformation-theoretic interpretation of $\dq$, as well as the more geometric aspects (in particular, we will do some computations and relate $\dq$ to the mutation-mutation autoequivalence).

\p  The author would like to thank his PhD supervisor, Jon Pridham, for his continued guidance throughout the project. He would like to thank Jenny August, Joe Chuang, Simon Crawford, Ben Davison, Zheng Hua, Andrey Lazarev, and Dong Yang for helpful discussions. He would especially like to thank Martin Kalck, Bernhard Keller, and Michael Wemyss for their valuable comments which improved both the exposition of the paper and the results obtained. Finally, he would like to thank the anonymous referee for their careful reading, helpful suggestions, and for pointing out an error in the proof of \ref{qisolem}.

\p The author would like to thank Hongxing Chen and Zhengfang Wang for pointing out a mistake in \ref{kymap} of the published version of this paper, and their helpful comments towards its rectification. This version rectifies that mistake. The author has endeavoured to keep the numbering of theorems, definitions etc.\ as close to the published version as possible. This has lead to some oddities: firstly, the logical place for the new section 6.5 to appear is directly after section 6.2, but it does not. Sections 8.1 and 8.2 contain some extraneous material that is correct but not particularly relevant. For a full list of differences between this version and the published paper, see the errata at \url{http://mattbooth.info/papers/singcats-errata.pdf}.

			\section{Notation and conventions}
	Throughout this paper, $k$ will denote an algebraically closed field of characteristic zero. Many of our theorems are true in positive characteristic, or even if one drops the algebraically closed assumption, and we will try to indicate where this holds. Modules are right modules by default. Consequently, noetherian means right noetherian, global dimension means right global dimension, et cetera. Unadorned tensor products are by default over $k$. We denote isomorphisms (of modules, functors, \ldots) with $\cong$ and weak equivalences with $\simeq$.
	
	\p If we refer to an object as just \textbf{graded}, then by convention we mean that it is $\Z$-graded. We use cohomological grading conventions, so that the differential of a complex has degree $1$. If $X$ is a complex, we will denote its cohomology complex by $H(X)$ or just $HX$. If $X$ is a complex, let $X[i]$ denote `$X$ shifted left $i$ times': the complex with $X[i]^j=X^{i+j}$ and the same differential as $X$, but twisted by a sign of $(-1)^i$. This sign flip can be worked out using the \textbf{Koszul sign rule}: when an object of degree $p$ moves past an object of degree $q$, one should introduce a factor of $(-1)^{pq}$. If $x$ is a homogeneous element of a complex of modules, we denote its degree by $|x|$.

	\p Recall that a $k$-algebra is a $k$-vector space with an associative unital $k$-bilinear multiplication. In other words, this is a monoid inside the monoidal category $(\cat{Vect}_k, \otimes)$. Similarly, a \textbf{differential graded algebra} (\textbf{dga} for short) over $k$ is a monoid in the category of chain complexes of vector spaces. More concretely, a dga is a complex of $k$-vector spaces $A$ with an associative unital chain map $\mu:A\otimes A \to A$, which we refer to as the multiplication. The condition that $\mu$ be a chain map forces the differential to be a derivation for $\mu$. 
	
	\p A $k$-algebra is equivalently a dga concentrated in degree zero, and a graded $k$-algebra is equivalently a dga with zero differential. We will sometimes refer to $k$-algebras as \textbf{ungraded algebras} to emphasise that they should be considered as dgas concentrated in degree zero. A dga is \textbf{graded-commutative} or just \textbf{commutative} or a \textbf{cdga} if all graded commutator brackets $[x,y]=xy-(-1)^{|x||y|}yx$ vanish. Commutative polynomial algebras are denoted with square brackets $k[x_1,\ldots, x_n]$ whereas noncommutative polynomial algebras are denoted with angle brackets $k\langle x_1,\ldots, x_n\rangle$. Note that in a cdga, even degree elements behave like elements of a symmetric algebra, whereas odd degree elements behave like elements of an exterior algebra: in particular, odd degree elements are square-zero since they must commute with themselves.
	
	\p A \textbf{dg module} (or just a \textbf{module}) over a dga $A$ is a complex of vector spaces $M$ together with an action map $M \otimes A \to M$ satisfying the obvious identities (equivalently, a dga map $A \to \enn_k(M)$). Note that a dg module over an ungraded ring is exactly a complex of modules. Just as for modules over a ring, the category of dg modules is a closed monoidal abelian category. If $A$ is an algebra, write $\cat{Mod}\text{-}A$ for its category of right modules and $\cat{mod}\text{-}A\subseteq\cat{Mod}\text{-}A$ for its category of finitely generated modules.
	
	\p Say that a complex $X$ is \textbf{connective} if one has $X^i=0$ for $i>0$. Up to quasi-isomorphism, by taking the good truncation to nonpositive degrees it is enough to assume that $H^i(X)\cong0$ for $i>0$. Say that $X$ is \textbf{coconnective} if one has $X^i=0$ for $i<0$ (or equivalently $H^i(X)\cong0$ for $i<0$ up to quasi-isomorphism). We use the same terminology in case $X$ admits extra structure (e.g. that of a dga). Note that if $A$ is a connective dga in the weak sense then the good truncation map $\tau_{\leq 0 }A \into A$ is a dga quasi-isomorphism. We use the term `connective' rather than `nonpositive' as the former notion is independent of our grading conventions.
	
	\p We freely use terminology and results from the theory of model categories; see \cite{quillenHA, hovey, dwyerspalinski, riehlCHT} for references. We will also assume that the reader has a good familiarity with the theory of triangulated and derived categories; see \cite{neemanloc} and \cite{weibel, huybrechts} respectively for references. In particular we will make use of the fact that the derived category of a dga is the homotopy category of a model category. By convention we use the projective model structure on dg-modules, where every object is fibrant, and over a ring the cofibrant complexes are precisely the perfect complexes (this is the `q-model structure' of \cite{sixmodels}).

	\section{The derived quotient}
		We begin by introducing our main object of study: the derived quotient of Braun--Chuang--Lazarev \cite{bcl}.  We will mostly be interested in derived quotients of ungraded algebras by idempotents. The derived quotient is a natural object to study, and has been investigated before by a number of authors: for example it appears in Kalck and Yang's work \cite{kalckyang,kalckyang2, kalckyang3} on relative singularity categories, de Thanhoffer de V\"olcsey and Van den Bergh's paper \cite{dtdvvdb} on stable categories, and Hua and Zhou's paper \cite{huazhou} on the noncommutative Mather--Yau theorem. Our study of the derived quotient will unify some of the aspects of all of the above work. We remark that all of the results of this section are valid over any field $k$.
	\subsection{Derived localisation}\label{derloc}
	The derived quotient is a special case of a general construction -- the derived localisation. Let $A$ be any dga over $k$ (the construction works over any commutative base ring). Let $S \subseteq H(A)$ be any collection of homogeneous cohomology classes. Braun, Chuang and Lazarev define the \textbf{derived localisation} of $A$ at $S$, denoted by $\dloc$, to be the dga universal with respect to homotopy inverting elements of $S$:
	\begin{defn}[{\cite[\S3]{bcl}}]
		Let $Q A \to A$ be a cofibrant replacement of $A$. The \textbf{derived under category} $A \downarrow^\mathbb{L} \cat{dga}$ is the homotopy category of the under category $Q A \downarrow \cat{dga}$ of dgas under $Q A$. A $QA$-algebra $f:Q A \to Y$ is \textbf{$S$-inverting} if for all $s \in S$ the cohomology class $f(s)$ is invertible in $HY$. The \textbf{derived localisation} $\dloc$ is the initial object in the full subcategory of $S$-inverting objects of $A \downarrow^\mathbb{L} \cat{dga}$.
	\end{defn} 
	\begin{prop}[{\cite[3.10, 3.4, and 3.5]{bcl}}]
		The derived localisation exists, is unique up to unique isomorphism in the derived under category, and is quasi-isomorphism invariant.
	\end{prop}
	In particular, the derived localisation $\dloc$ comes with a canonical map from $A$ (in the derived under category) making it into an $A$-bimodule, unique up to $A$-bimodule quasi-isomorphism. In what follows, we will refer to $\dloc$ as an $A$-bimodule, with the assumption that this always refers to the canonical bimodule structure induced from the dga map $A \to \dloc$.
	\begin{rmk}
		The derived localisation is the homotopy pushout of the span$$A \from k\langle S \rangle \to k\langle S, S^{-1}\rangle.$$
	\end{rmk}
	
	\begin{defn}A map $A \to B$ of dgas is a \textbf{homological epimorphism} if the multiplication map $B\lot_A B \to B$ is a quasi-isomorphism of $B$-modules.
	\end{defn}
	\begin{prop}\label{homepi}
		Let $A$ be a dga and $S \subseteq H(A)$ be any collection of homogeneous cohomology classes. Then the canonical localisation map $A \to \dloc$ is a homological epimorphism.
	\end{prop}
	\begin{proof}
		The map is a homotopy epimorphism by \cite[3.17]{bcl}, which by \cite[4.4]{clhomepi} is the same as a homological epimorphism.
	\end{proof}

	\begin{defn}[{\cite[4.2 and 7.1]{bcl}}]
		Let $X$ be an $A$-module. Say that $X$ is \textbf{$S$-local} if, for all $s\in S$, the map $s: X \to X$ is a quasi-isomorphism. Say that $X$ is \textbf{$S$-torsion} if $\R\hom_A(X,Y)$ is acyclic for all $S$-local modules $Y$. Let $D(A)_{S\text{-loc}}$ be the full subcategory of $D(A)$ on the $S$-local modules, and let $D(A)_{S\text{-tor}}$ be the full subcategory on the $S$-torsion modules.
	\end{defn}
	Similarly as for algebras, one defines the notion of the derived localisation $\mathbb{L}_S(X)$ of an $A$-module $X$. It is not too hard to prove the following:
	\begin{thm}[{\cite[4.14 and 4.15]{bcl}}]
		Localisation of modules is smashing, in the sense that $X \to X \lot_A \dloc$ is the derived localisation of $X$. Moreover, restriction of scalars gives an equivalence of $D(\dloc)$ with $D(A)_{S\text{-loc}}$.
	\end{thm}
In particular, the dga $\dloc$ is the derived localisation of the $A$-module $A$.
	\begin{rmk}
		If $A\to B$ is a dga map then the three statements 
		\begin{itemize}
			\item $A \to B$ is a homological epimorphism.
			\item $A \to B$ induces an embedding $D(B)\to D(A)$.
			\item $-\lot_AB$ is a smashing localisation on $D(A)$.
		\end{itemize}
		are all equivalent \cite{phomepis}.
	\end{rmk}

	One defines a \textbf{colocalisation functor} pointwise by setting $\mathbb{L}^S(X)\coloneqq \mathrm{cocone}(X \to \mathbb{L}_S(X))$. An easy argument shows that $\mathbb{L}^S(X)$ is $S$-torsion. 
	\begin{defn}\label{colocdga}The \textbf{colocalisation} of $A$ along $S$ is the dga $$\mathbb{L}^S(A)\coloneqq \R\enn_A\left(\oplus_{s \in S}\,\mathrm{cone}(A \xrightarrow{s} A)\right).$$
	\end{defn}Note that the dga $\mathbb{L}^S(A)$ may differ from the colocalisation of the $A$-module $A$. If $S$ is a finite set, then the dga $\mathbb{L}^S(A)$ is a compact $A$-module, and we get the analogous:
	\begin{thm}[{\cite[7.6]{bcl}}]
		Let $S$ be a finite set. Then $D(\mathbb{L}^S(A))$ and $D(A)_{S\text{-tor}}$ are equivalent.
	\end{thm}
	Neeman--Thomason--Trobaugh--Yao localisation gives the following:
	\begin{thm}[{\cite[7.3]{bcl}}]\label{ntty}
		Let $S$ be finite. Then there is a sequence of dg categories $$\per\mathbb{L}^S(A) \to \per A \to \per \dloc$$which is exact up to direct summands.
	\end{thm}
	\begin{rmk}[{\cite[7.9]{bcl}}]\label{bclrcl}
		The localisation and colocalisation functors fit into a recollement $$\begin{tikzcd}[column sep=huge]
		D(A)_{S\text{-loc}} \ar[r]& D(A)\ar[l,bend left=25]\ar[l,bend right=25]\ar[r] & D(A)_{S\text{-tor}}\ar[l,bend left=25]\ar[l,bend right=25]
		\end{tikzcd}$$We will see a concrete special case of this in \ref{recoll}.
	\end{rmk}

	\begin{defn}[{\cite[9.1 and 9.2]{bcl}}]
		Let $A$ be a dga and let $e$ be an idempotent in $H^0(A)$. The \textbf{derived quotient} $\dq$ is the derived localisation $\mathbb{L}_{1-e}A$.
	\end{defn}
	Clearly, $\dq$ comes with a natural quotient map from $A$. One can write down an explicit model for $\dq$, at least when $k$ is a field.
	\begin{prop}\label{drinfeld}
		Let $A$ be a dga over $k$, and let $e\in H^0(A)$ be an idempotent. Then the derived quotient $\dq$ is quasi-isomorphic as an $A$-dga to the dga $$B\coloneqq \frac{A\langle h \rangle}{(he=eh=h)}\;, \quad d(h)=e$$with $h$ in degree -1.
	\end{prop}
	\begin{proof}
		This is essentially \cite[9.6]{bcl}; because $k$ is a field, $A$ is flat (and in particular left proper) over $k$. The quotient map $A \to B$ is the obvious one.
	\end{proof}
	\begin{rmk}This specific model for $\dq$ is an incarnation of the Drinfeld quotient: see \cite[9.7]{bcl} for the details.
	\end{rmk}
	\begin{rmk}
		In particular, this is a concrete model for $\dq$ as an $A$-bimodule.
	\end{rmk}
	\subsection{Cohomology}
	Let $A$ be a dga and let $e\in A$ be an idempotent. Write $R$ for the cornering\footnote{We take this terminology from \cite{cikcorner}; the motivating example is to take one of the obvious nontrivial idempotents in $M_2(k)$ to obtain a subalgebra (isomorphic to $k$) on matrices with entries concentrated in one corner.} $eAe$. We will investigate the cohomology of the derived quotient $Q\coloneqq \dq$.
	\begin{defn}[{cf.\ Dwyer and Greenlees \cite{dgcompletetorsion}}]\label{celldef}
		The \textbf{cellularisation functor}, denoted by $\cell : D(A) \to D(A)$, is the functor that sends $M$ to $Me\lot_R eA$.
	\end{defn}
	In particular, the cellularisation of $A$ itself is the bimodule $Ae\lot_R eA$. Note that this admits an $A$-bilinear multiplication map $\mu: \cell A \to A$ which has image the submodule $AeA\into A$.
	\begin{prop}[{cf.\ \cite[\S 7]{kalckyang2} and \cite[\S4]{nsparam}}]\label{dqexact}
		The multiplication map $\cell A \xrightarrow{\mu} A$ and the structure map $A \to Q$ fit into an exact triangle of $A$-bimodules $\cell A \xrightarrow{\mu} A \to Q \to $.
	\end{prop}
	\begin{proof}Forget the algebra structure on $Q$ and view it as an $A$-bimodule; recall that the localisation map $A \to Q$ is the derived localisation of the $A$-module $A$. Observe that the localisation map $A \to Q$ is also the localisation of the $A$-module $A$ at the perfect module $Ae$, in the sense of \cite{dgcompletetorsion}. Thus by \cite[4.8]{dgcompletetorsion}, the homotopy fibre of $A \to Q$ is the cellularisation of $A$.
	\end{proof}
	When $A$ is an ungraded algebra, then one can write down a much more explicit proof using the `Drinfeld quotient' model for $Q$. We do this below, since we will need to use some facts about this explicit model later.
	\begin{lem}\label{drinfeldmodel}
		Suppose that $A$ is an ungraded algebra with an idempotent $e\in A$. Put $R\coloneqq eAe$ the cornering. Let $B$ be the dga of \ref{drinfeld} quasi-isomorphic to $Q$. Then:
		\begin{enumerate}
			\item Let $n>0$ be an integer. There is an $A$-bilinear isomorphism $$B^{-n}\cong Ae \otimes R\otimes\cdots \otimes R \otimes eA$$where the tensor products are taken over $k$ and there are $n$ of them.
			\item Let $n>0$. The differential $B^{-n} \to B^{-n+1}$ is the Hochschild differential, which sends $$x_0\otimes \cdots \otimes x_n\mapsto\sum_{i=0}^{n-1} (-1)^i x_0\otimes\cdots \otimes x_ix_{i+1}\otimes \cdots \otimes x_n.$$
			\item Let $n,m>0$ and let $a\in B^0=A$. Let $x=x_0\otimes \cdots \otimes x_n \in B^{-n}$ and $y=y_0\otimes \cdots \otimes y_m\in B^{-m}$. Then we have
			\begin{align*}
			xy &= x_0\otimes \cdots \otimes x_n y_0\otimes \cdots \otimes y_m
			\\ ax&=ax_0\otimes \cdots \otimes x_n
			\\ xa&=x_0\otimes \cdots \otimes x_na.
			\end{align*}
		\end{enumerate}	
	\end{lem}
	\begin{proof}
		For the first claim, observe that a generic element of $B^{-n}$ looks like a path $x_0h\cdots hx_n$ where $x_0=x_0e$, $x_n=ex_n$, and $x_j=ex_je$ for $0<j<n$. Replacing occurrences of $h$ with tensor product symbols gives the claimed isomorphism. For the second claim, because $h$ has degree $-1$ we must have $$d(x_0h\cdots hx_n)=\sum_i (-1^i)x_0h\cdots hx_id(h)x_{i+1}h \cdots hx_n$$ but because $d(h)=e$ we have $x_id(h)x_{i+1}=x_ix_{i+1}$. The third claim is clear from the definition of $B$.
	\end{proof}

	\begin{proof}[Alternate proof of \ref{dqexact} when $A$ is ungraded]Consider the shifted bimodule truncation $$T:=(\tau_{\leq-1}B)[-1] \simeq \cdots \to Ae\otimes R\otimes R\otimes eA\to Ae\otimes R\otimes eA\to Ae\otimes eA$$with $Ae\otimes eA$ in degree zero. By \ref{drinfeldmodel}(2) we see that this truncation is exactly the complex that computes the relative Tor groups $\tor^{R/k}(Ae,eA)$ \cite[8.7.5]{weibel}. Since $k$ is a field, the relative Tor groups are the same as the absolute Tor groups, and hence $T$ is quasi-isomorphic to $\cell A$. Because we have $B\simeq \mathrm{cone}(T \xrightarrow{\mu} A)$ we are done.
	\end{proof}
	The following is immediately obtained by considering the long exact sequence associated to the exact triangle $\cell A \xrightarrow{\mu} A \to Q \to $.
	\begin{cor}\label{derquotcohom}
		Let $A$ be an algebra over a field $k$, and let $e\in A$ be an idempotent.
		Then the derived quotient $\dq$ is a connective dga, with cohomology spaces 
		$$H^j(\dq)\cong\begin{cases}
		0 & j>0
		\\ A/AeA & j=0
		\\ \ker(Ae\otimes_{R}eA \to A) & j=-1
		\\ \tor^{R}_{-j-1}(Ae,eA) & j<-1
		\end{cases}$$
	\end{cor}
	\begin{rmk}
		The ideal $AeA$ is said to be \textbf{stratifying} if the map $Ae \lot_{R} eA \to AeA$ is a quasi-isomorphism. It is easy to see that $AeA$ is stratifying if and only if $H^0:\dq \to A/AeA$ is a quasi-isomorphism.
	\end{rmk}

	\begin{ex}\label{quiv1} Let $A$ be the path algebra (over $k$) of the quiver $$
		
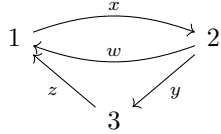
\begin{tikzcd}
		1 \arrow[rr,bend left=20,"x"]  && 2 \arrow[ll,bend left=20,"w"']\ar[ld,"y"] \\ & 3 \ar[lu,"z"]&
		\end{tikzcd}$$ modulo the relations $w=yz$ and $xyz=yzx=zxy=0$. Put $e\coloneqq e_1+e_2 \in A$ and put $R\coloneqq eAe$. It is not hard to compute that $\mathrm{dim}_k(A)=9$, $\mathrm{dim}_k(R)=4$, and $\mathrm{dim}_k(A/AeA)=1$. One can check using \ref{derquotcohom} that $H^{-1}(\dq)$ is two-dimensional, with basis $\{e\otimes w - y \otimes z, z \otimes xy\}$. Loosely, $H^{-1}(\dq)$ measures how many relations there are in $A$ that cannot be `seen' from the vertex set $\{1,2\}$.
	\end{ex}

	\subsection{Recollements}
	Loosely speaking, a recollement (see \cite{bbd} or \cite{jorgensen} for a definition) between three triangulated categories $(\mathcal{T}',\mathcal{T},\mathcal{T}'')$ is a collection of functors describing how to glue $\mathcal{T}$ from a subcategory $\mathcal{T}'$ and a quotient category $\mathcal{T}''$. One can think of a recollement as a short exact sequence $\mathcal{T}' \to \mathcal{T} \to \mathcal{T}''$ of triangulated categories where both maps admit left and right adjoints.
	\begin{thm}[{cf.\ \cite[2.10]{kalckyang} and \cite[9.5]{bcl}}]\label{recoll}
		Let $A$ be an algebra over $k$, and let $e\in A$ be an idempotent. Write $Q\coloneqq \dq$ and $R\coloneqq eAe$. Let $Q_A$ denote the $Q\text{-}A$-bimodule $Q$, let ${}_AQ$ denote the $A\text{-}Q$-bimodule $Q$, and let ${}_AQ_A$ denote the $A$-bimodule $Q$. Put \begin{align*}
		i^*\coloneqq -\lot_A {}_AQ, &\quad j_!\coloneqq  -\lot_{R} eA
		\\ i_*=\R\hom_{Q}({}_AQ,-), &\quad j^!\coloneqq \R\hom_A(eA,-)
		\\ i_!\coloneqq \lot_{Q}Q_A, & \quad j^*\coloneqq -\lot_A Ae
		\\ i^! \coloneqq \R\hom_{A}(Q_A,-), & \quad j_*\coloneqq \R\hom_{R}(Ae,-)
		\end{align*}
		Then the diagram of unbounded derived categories
		$$\begin{tikzcd}[column sep=huge]
		D(Q) \ar[r,"i_*=i_!"]& D(A)\ar[l,bend left=25,"i^!"']\ar[l,bend right=25,"i^*"']\ar[r,"j^!=j^*"] & D(R)\ar[l,bend left=25,"j_*"']\ar[l,bend right=25,"j_!"']
		\end{tikzcd}$$
		is a recollement diagram.
	\end{thm}
	\begin{proof}We give a rather direct proof. It is clear that $(i^*,i_*=i_!,i^!)$ and $(j_!, j^!=j^*,j_*)$ are adjoint triples, and that $i_*=i_!$ is fully faithful. Fullness and faithfulness of $j_!$ and $j_*$ follow from \cite[2.10]{kalckyang}. The composition $j^*i_*$ is tensoring by the $Q\text{-}R$-bimodule $Q.e$, which is acyclic since $Q$ is $e$-killing in the sense of \cite[\S9]{bcl}. The only thing left to show is the existence of the two required classes of distinguished triangles. First observe that 
		\begin{align*}& i_!i^!\cong \R\hom_A({}_AQ_A,-)
		\\ & j_*j^* \cong \R\hom_{R}(Ae,\R\hom_A(eA,-))\cong \R\hom_A(\cell A,-)
		\\ & j_!j^! \cong -\lot_A \cell A
		\\ &  i_*i^* \cong -\lot_A {}_AQ_A\end{align*}
		Now, recall from \ref{dqexact} the existence of the distinguished triangle of $A$-bimodules $$\cell A \xrightarrow{\mu} A \xrightarrow{} {}_AQ_A \to$$
		Taking any $X$ in $D(A)$ and applying $\R\hom_A(-,X)$ to this triangle, we obtain a distinguished triangle of the form $i_!i^!X \to X \to j_*j^* X \to$. Similarly, applying $X\lot_A-$, we obtain a distinguished triangle of the form $j_!j^!X \to X \to i_*i^* X \to$.
	\end{proof}
	\begin{rmk}\label{recolrmk1}This recollement is given in \cite[9.5]{bcl}, although they are not explicit with their functors. The existence of a dga $Q$ fitting into the above recollement is shown in \cite[2.10]{kalckyang}; our method has the advantage of being an explicit construction as well as being functorial with respect to maps of rings with idempotents. We remark that the existence of a dga quasi-isomorphic to $\dq$ fitting into a recollement as above already appears as \cite[7.1]{kalckyang2} (see also \cite[proof of Theorem 4]{nsparam}). If $AeA$ is stratifying, this recovers a recollement constructed by Cline, Parshall, and Scott \cite{cps}. See e.g.\ \cite{cikcorner} or \cite{pssrecol} for the analogous recollement on the level of abelian categories.
	\end{rmk}

	\begin{prop}\label{Rcoloc}
		In the above setup, $D(R)$ is equivalent to the derived category of $(1-e)$-torsion modules.
	\end{prop}
	\begin{proof}
		Recollements are determined completely by fixing one half (e.g.\ \cite[Remark 2.4]{kalck}). Now the result follows from the existence of the recollement of \ref{bclrcl}. More concretely, one can check that the colocalisation $\mathbb{L}^{1-e}A$ is $R$: because $A\cong eA\oplus(1-e)A$ as right $A$-modules, we have $\mathrm{cone}(A \xrightarrow{1-e} A)\simeq eA$, and we know that $\R\enn_A(eA)\simeq \enn_A(eA)\cong R$ because $eA$ is a projective $A$-module.
	\end{proof}
\begin{rmk}\label{derivedmoritarmk}
	In particular, if $Q'$ is any other $A$-algebra fitting into a recollement as in \ref{recoll}, then $Q'$ must be derived Morita equivalent to $Q$. A sharper uniqueness statement appears as \cite[6.2]{kalckyang2}.
	\end{rmk}
	We show that $\dq$ is a relatively compact $A$-module; before we do this we first introduce some notation.
	\begin{defn}Let $\mathcal X$ be a subclass of objects of a triangulated category $\mathcal{T}$. Then $\thick_\mathcal{T} \mathcal X $ denotes the smallest triangulated subcategory of $\mathcal{T}$ containing $\mathcal{X}$ and closed under taking direct summands. Similarly, $\langle \mathcal X \rangle_\mathcal{T}$ denotes the smallest triangulated subcategory of $\mathcal{T}$ containing $\mathcal{X}$, and closed under taking direct summands and all existing set-indexed coproducts. We will often drop the subscripts if $\mathcal T$ is clear. If $\mathcal X$ consists of a single object $X$, we will write $\thick X$ and $\langle X \rangle$.
	\end{defn}
	\begin{ex}
		Let $A$ be a dga. Then $\langle A \rangle_{D(A)}\cong D(A)$, whereas $\thick_{D(A)}(A) \cong \per A$.
	\end{ex}
	\begin{defn}
		Let $\mathcal{T}$ be a triangulated category and let $X$ be an object of $\mathcal{T}$. Say that $X$ is \textbf{relatively compact} (or \textbf{self compact}) in $\mathcal{T}$ if it is compact as an object of $\langle X \rangle_{\mathcal{T}}$.
	\end{defn}
	\begin{prop}[{cf.\ \cite[3.3]{jorgensen}}]
		The right $A$-module $\dq$ is relatively compact in $D(A)$.
	\end{prop}
	\begin{proof}The embedding $i_*$ is a left adjoint and so respects coproducts. Hence $i_*(\dq)$ is relatively compact in $D(A)$ by \cite[1.7]{jorgensen}. The essential idea is that $\dq$ is compact in $D(\dq)$, the functor $i_*$ is an embedding, and $\langle i_*(\dq) \rangle \subseteq \im i_*$.
	\end{proof}
	In situations when $\dq$ is not a compact $A$-module (e.g.\ when it has nontrivial cohomology in infinitely many degrees), this gives interesting examples of relatively compact objects that are not compact.
	\begin{defn}
		Let $D(A)_{A/AeA}$ denote the full subcategory of $D(A)$ on those modules $M$ with each $H^j(M)$ a module over $A/AeA$.
	\end{defn}
	\begin{prop}\label{cohomsupport}There is a natural triangle equivalence $D(\dq)\cong D(A)_{A/AeA}$.
	\end{prop}
	\begin{proof}Follows from the proof of \cite[2.10]{kalckyang}, along with the fact that recollements are determined completely by fixing one half.
	\end{proof}
	\begin{prop}\label{semiorthog}
		The derived category $D(A)$ admits a semi-orthogonal decomposition $$D(A)\cong \langle D(A)_{A/AeA}, \langle eA \rangle \rangle = \langle \im i_*, \im j_! \rangle $$
	\end{prop}
	\begin{proof}
		This is an easy consequence of \cite[3.6]{jorgensen}.
	\end{proof}
	We finish with a couple of facts about t-structures; see \cite{bbd} for a definition. In particular, given t-structures on the outer pieces of a recollement diagram, one can glue them to a new t-structure on the central piece \cite[1.4.10]{bbd}. 
	\begin{thm}\label{tstrs}
		The category $D(\dq)$ admits a t-structure $\tau$ with aisles $$\tau^{\leq 0}=\{X: H^i(X)=0 \text{ for } i>0\}\qquad \text{and} \qquad \tau^{\geq 0}=\{X: H^i(X)=0 \text{ for } i<0\}.$$Moreover, $H^0$ is an equivalence from the heart of $\tau$ to $\cat{Mod}$-$A/AeA$. Furthermore, gluing $\tau$ to the natural t-structure on $D(R)$ via the recollement diagram of \ref{recoll} gives the natural t-structure on $D(A)$.
	\end{thm}
	\begin{proof}
		The first two sentences are precisely the content of \cite[2.1(a)]{kalckyang}. The last assertion holds because gluing of t-structures is unique, and restricting the natural t-structure on $D(A)$ clearly gives $\tau$ along with the natural t-structure on $D(R)$.
	\end{proof}

	\subsection{Hochschild theory}
	We collect some facts about the Hochschild theory of the derived quotient. The most important is that taking quotients preserves Hochschild (co)homology complexes:
	\begin{prop}[{\cite[6.2]{bcl}}]\label{hochprop}
		Let $A$ be a dga, $e\in H^0(A)$ an idempotent, and $Q\coloneqq \dq$ the derived quotient. Let $M$ be a $Q$-module. Then there are quasi-isomorphisms $$Q \lot_{Q^e} M \simeq A \lot_{A^e} M \quad\text{and}\quad \R\hom_{Q^e}(Q,M) \simeq \R\hom_{A^e}(A,M) $$and hence isomorphisms $$HH_*(Q,M)\cong HH_*(A,M) \quad\text{and}\quad HH^*(Q,M)\cong HH^*(A,M).$$
	\end{prop}
	Hochschild homology is functorial with respect to recollement:
	\begin{prop}[{\cite[3.1]{kellerhoch}}]
		Let $A$ be an algebra over $k$ and $e\in A$ an idempotent. Put $Q\coloneqq \dq$ the derived quotient and $R\coloneqq eAe$ the cornering. Then there is an exact triangle in $D(k)$ $$
		Q\lot_{Q^e} Q \to A \lot_{A^e} A \to R \lot_{R^e} R \to .$$
	\end{prop}
	Unfortunately, Hochschild cohomology does not behave so nicely.
	\begin{prop}Let $A$ be a algebra over $k$ and $e\in A$ an idempotent. Put $Q\coloneqq \dq$ the derived quotient and $R\coloneqq eAe$ the cornering. Then there are exact triangles in $D(k)$\begin{align*}
		\R\hom_{A^e}(Q,A) \to \R\hom_{A^e}(A,A) \to \R\hom_{R^e}(R,R) \to \\
		\R\hom_{A^e}(A,\cell A) \to \R\hom_{A^e}(A,A) \to \R\hom_{Q^e}(Q,Q) \to .
		\end{align*}
	\end{prop}
	\begin{proof}This follows directly from \cite[Theorem 4]{han}; the proof in our setting is not difficult, so we give a full argument following the proof given there. Note that the idea of using the $3\times3$ square already appears in the proof of \cite[2.1]{koenignagase}. Recall that $\mathrm{cocone}(A \to Q)$ is quasi-isomorphic as an $A$-bimodule to $\cell A$ by \ref{dqexact}. Consider the diagram $$\begin{tikzcd} \phantom{}& \phantom{}& \phantom{}& \phantom{}\\
		\R\hom_{A^e}(\cell A,\cell A) \ar[r]\ar[u] & \R\hom_{A^e}(\cell A,A) \ar[r]\ar[u] & \R\hom_{A^e}(\cell A,Q) \ar[r]\ar[u] & \phantom{} \\
		\R\hom_{A^e}(A,\cell A) \ar[r]\ar[u] & \R\hom_{A^e}(A,A) \ar[r]\ar[u] & \R\hom_{A^e}(A,Q) \ar[r]\ar[u] & \phantom{} \\
		\R\hom_{A^e}(Q,\cell A) \ar[r]\ar[u] & \R\hom_{A^e}(Q,A) \ar[r]\ar[u] & \R\hom_{A^e}(Q,Q) \ar[r]\ar[u,"f"] & \phantom{} \end{tikzcd}$$ whose rows and columns are exact triangles. The first triangle can be seen as the middle column, once we make the observation that $$\R\hom_{A^e}(\cell A,A)\simeq \R\hom_{A^{\mathrm{op}}\otimes R}(Ae,\R\hom_A(eA,A))\simeq\R\hom_{A^{\mathrm{op}}\otimes R}(Ae,Ae)\simeq \R\hom_{R^e}(R,R).$$Now, \ref{hochprop} tells us both that the labelled arrow $f$ is a quasi-isomorphism, and moreover that both source and target are quasi-isomorphic to $\R\hom_{Q^e}(Q,Q)$. The second triangle is now visible as the middle row.
	\end{proof}
	\begin{rmk}\label{recolrmk2}We have recovered two of the three triangles obtained by applying \cite[Theorem 4]{han} to the standard recollement $(D(Q),D(A), D(R))$. The third is $$\R\hom_{A^e}(Q,\cell A) \to \R\hom_{A^e}(A,A) \to \R\hom_{Q^e}(Q,Q) \oplus \R\hom_{R^e}(R,R) \to.$$When $AeA$ is a stratifying ideal, these three exact triangles recover the long exact sequences of Koenig--Nagase \cite[3.4]{koenignagase}.
	\end{rmk}
	\section{DG singularity categories}
	First studied by Buchweitz \cite{buchweitz} for noncommutative rings, and then by Orlov \cite{orlovtri} for schemes, the singularity category of a geometric object is a measure of how singular it is: in reasonable cases, it vanishes if and only if the object in consideration is smooth.
We begin with some recollections on dg categories, before defining singularity categories and their dg enhancements. We remark that dg singularity categories have already been studied in the setting of derived geometry by Blanc--Robalo--To\"en--Vezzosi \cite{motivicsingcat}. We also define the dg relative singularity category; aside from being a convenient tool for us, relative singularity categories are very useful in constructing nontrivial equivalences between singularity categories \cite{kkknoerrer}.
	
	\subsection{DG categories}\label{dgcats}
	We assume that the reader is familiar with the basics of the theory of dg categories; survey articles on dg categories include \cite{keller, toendglectures}. We use this section primarily to set notation. All dg categories will be linear over our base field $k$.
	
	Recall that a ($k$-linear) \textbf{dg category} is a category enriched over the symmetric monoidal category of complexes of vector spaces, and that a \textbf{dg functor} between two categories is an enriched functor. If $\mathcal{C}$ is a dg category, and $x,y$ are two objects of $\mathcal{C}$, we denote the hom-complex of maps from $x$ to $y$ by $\dgh_\mathcal{C}(x,y)$. Similarly we denote the endomorphism dga of $x$ by $\dge_\mathcal{C}(x)$. We may omit the subscript of $\mathcal{C}$ if the context is clear. We use this notation because we will want to use $\underline{\hom}$ to denote the homsets in the stable category of a ring. Observe that a dg functor $F:\mathcal{C}\to\mathcal{D}$ induces dga morphisms $F_{xx}:\dge_{\mathcal{C}}(x)\to\dge_{\mathcal{D}}(Fx)$ for every $x \in \mathcal{C}$. 
	
	We denote the homotopy category of $\mathcal{C}$ by $[\mathcal{C}]$; this is the $k$-linear category with the same objects as $\mathcal{C}$ and hom-spaces are given by $\hom_{[\mathcal{C}]}(x,y)\coloneqq H^0(\dgh_\mathcal{C}(x,y))$. We sometimes write $[x,y]\coloneqq \hom_{[\mathcal{C}]}(x,y)$. Say that a dg functor $F:\mathcal{C}\to\mathcal{D}$ is \begin{itemize}
			\item $F$ is \textbf{quasi-fully faithful} if all of its components $F_{xy}$ are quasi-isomorphisms.
			\item $F$ is \textbf{quasi-essentially surjective} if the induced functor $[F]:[\mathcal{C}]\to[\mathcal{D}]$ is essentially surjective.
			\item $F$ is a \textbf{quasi-equivalence} if it is quasi-fully faithful and quasi-essentially surjective.
		\end{itemize}
	
	In a dg category, one may define shifts and mapping cones via the Yoneda embedding into the category of modules. This is equivalent to defining them as representing objects of the appropriate functors; e.g.\ $x[1]$ should represent $\dgh(x,-)[-1]$. Say that a dg category is \textbf{pretriangulated} if it contains a zero object and is closed under shifts and mapping cones. If $\mathcal{C}$ is pretriangulated then the homotopy category $[\mathcal{C}]$ is canonically triangulated, with translation functor given by the shift. 
	
	If $A$ is a dga, then $D_\mathrm{dg}(A)$ denotes the pretriangulated dg category of cofibrant dg modules over $A$, and $\cat{per}_\mathrm{dg}(A) \subseteq {D}_\mathrm{dg}(A)$ denotes the pretriangulated dg subcategory of compact objects. In the notation of \cite{toendglectures}, $\cat{per}_\mathrm{dg}(A)$ is $\hat{A}_{\text{pe}}$. In addition, if $A$ is a $k$-algebra then $D^b_\mathrm{dg}(A)$ denotes the dg category of cofibrant dg $A$-modules with bounded cohomology; these are precisely the bounded above complexes of projective $A$-modules with bounded cohomology.
	
	One has equivalences of triangulated categories $[D_\mathrm{dg}(A)]\cong D(A)$, ${[D^b_\mathrm{dg}(A)]\cong D^b(A)}$ and $[\cat{per}_\mathrm{dg}(A)]\cong\cat{per}(A)$. Note that in the dg categories above, $\dgh$ is a model for the derived hom $\R\hom$; we will implicitly use this fact often.
	One can invert quasi-equivalences between dg categories: 
	\begin{thm}[Tabuada \cite{tabuadamodel}]\label{tabmod}
		The category of all small dg categories admits a (cofibrantly generated) model structure where the weak equivalences are the quasi-equivalences. The fibrations are the objectwise levelwise surjections that lift isomorphisms. Every dg category is fibrant.
	\end{thm}
	See \cite{tabuadamodel} for a more precise description of the model structure. The advantage of this result is that it gives one good control over $\mathrm{Hqe}$, the category of dg categories localised at the quasi-equivalences.
	
	\p If $R$ is a ring (or a dga) then the hom-complexes in the dg category $C_\mathrm{dg}(R)$ of dg-$R$-modules will be denoted by simply $\dgh_R$, as opposed to the lengthier $\dgh_{C_\mathrm{dg}(R)}$.

	\subsection{DG quotients}\label{dgquotsctn}
	One can take quotients of dg categories by dg subcategories; these dg quotients were first considered by Keller \cite{kellerdgquot}, and an explicit construction using ind-categories was given by Drinfeld \cite{drinfeldquotient}, which we recall in this section. We begin with the definition of an ind-category.
	
	\begin{defn}
		Let $\mathcal{C}$ be a category. An \textbf{ind-object} of $\mathcal C$ is a formal filtered colimit in $\mathcal{C}$; i.e. a functor $J \to \mathcal{C}$ where $J$ is a small filtered category. We denote such an ind-object by $\{C_j\}_{j \in J}$. The \textbf{ind-category} of $\mathcal{C}$ is the category $\cat{ind}\mathcal{C}$ whose objects are the ind-objects of $\mathcal{C}$ and whose morphisms are given by $$\hom_{\cat{ind}\mathcal{C}}(\{C_i\}_{i \in I}, \{D_j\}_{j \in J}) \coloneqq  \varprojlim_i \varinjlim_j \hom_{\mathcal{C}}(C_i,D_j)$$
	\end{defn}
	If $\mathcal{C}$ has filtered colimits, then there is a `realisation' functor $\varinjlim: \cat{ind}\mathcal{C} \to \mathcal{C}$. In this situation, if $D \in \mathcal{C}$ is a constant ind-object then one has $\hom_{\cat{ind}\mathcal{C}}(\{C_i\}_{i \in I}, D)\cong \hom_\mathcal{C}( \varinjlim_iC_i,D)$.  Note that if $\mathcal{C}$ is a dg category then so is $\cat{ind}\mathcal{C}$ in a natural way.
	\begin{defn}[Drinfeld \cite{drinfeldquotient}]
		Let $\mathcal{A}$ be a dg category and $\mathcal{B}\into\mathcal A$ a full dg subcategory. The \textbf{Drinfeld quotient} $\mathcal{A}/\mathcal{B}$ is the subcategory of $\cat{ind}\mathcal{A}$ on those $X$ such that:\begin{enumerate}
			\item $\dgh_{\cat{ind}\mathcal{A}}(\mathcal{B},X)$ is acyclic.
			\item There exists $a \in \mathcal{A}$ and a map $f:a \to X$ with $\mathrm{cone}(f)\in \cat{ind}\mathcal{B}$.
		\end{enumerate}
		Since $\cat{ind}\mathcal{A}$ is a dg category, so is $\mathcal{A}/\mathcal{B}$. The Drinfeld quotient is a model for ``the'' dg quotient:
		\begin{thm}[{\cite[4.02]{tabquot}}]
			Let $\mathcal{A}$ be a dg category and $i:\mathcal{B}\into \mathcal{A}$ a full dg subcategory. Then the quotient $\mathcal{A}/\mathcal{B}$ is the homotopy cofibre of $i$, taken in $\mathrm{Hqe}$.
		\end{thm}
		With this in mind, we will use the terms `Drinfeld quotient' and `dg quotient' interchangeably, although the careful reader should keep in mind that the former is merely a model for the latter, which exists only in a homotopical sense. For pretriangulated dg categories, the Drinfeld quotient is a dg enhancement of the Verdier quotient:
	\end{defn}
	\begin{thm}[{\cite[3.4]{drinfeldquotient}}]\label{drinfeldpretr}
		Let $\mathcal{A}$ be a pretriangulated dg category and $\mathcal{B}\into \mathcal{A}$ a full pretriangulated dg subcategory. Then there is a triangle equivalence $[\mathcal{A}/\mathcal{B}]\cong [\mathcal{A}]/[\mathcal{B}]$.
	\end{thm}
	
	\subsection{Triangulated and dg singularity categories}
We introduce the singularity category of a noncommutative ring, and enhance it to a dg category. We also introduce the relative singularity category and its natural dg enhancement. Let $R$ be any noetherian ring. Observe that the triangulated category $\per R$ of perfect complexes of $R$-modules embeds into the triangulated category $D^b(R)$ of all complexes of finitely generated modules with bounded cohomology.
	\begin{defn}
		Let $R$ be a noetherian ring. The \textbf{singularity category} of $R$ is the Verdier quotient $D_\mathrm{sg}R\coloneqq D^b(R)/ \per R$.
	\end{defn}
	\begin{rmk}
		Strictly, for noncommutative rings one should distinguish between the left and right singularity categories. However, we will always work with right modules.
	\end{rmk}
The following is a standard result about global dimension, regularity, and the singularity category; we provide a proof for completeness.
\begin{prop}
	Let $R$ be a commutative noetherian ring. Then the following hold:
	\begin{enumerate}
		\item If $R$ has finite global dimension, then $D_\mathrm{sg}(R)$ vanishes.
		\item $D_\mathrm{sg}(R)$ vanishes if and only if $R$ is regular.
		\item If $R$ is regular then the global dimension of $R$ is equal to its Krull dimension.
		\end{enumerate}
	\end{prop}
\begin{proof}
	 First we note that if every finitely generated $R$-module has finite projective dimension, then $D_\mathrm{sg}(R)$ vanishes: to see this, take a bounded complex $X\ = \ X_p \to \cdots \to X_q$ of finitely generated modules and write it as an iterated cone of maps between modules. By assumption each of these modules is quasi-isomorphic to a perfect complex. Taking mapping cones preserves perfect complexes, and hence $X$ must itself be quasi-isomorphic to a perfect complex. Hence $X$ maps to zero in the Verdier quotient and so $D_\mathrm{sg}(R)$ vanishes. Now statement (1) is clear. Moreover the `if' part of statement (2) now follows from the global Auslander--Buchsbaum--Serre theorem \cite[5.94]{lamlect}, as does statement (3). For the `only if' part of statement (2), note that if $\mathfrak{m}$ is a maximal ideal of $R$, then $\mathrm{gl.dim}(R_\mathfrak{m})=\mathrm{pd}_R(R/\mathfrak{m})$ \cite[5.92]{lamlect}. But $D_\mathrm{sg}(R)\cong 0$, so that every finitely generated $R$-module has finite projective dimension, and in particular each $R_\mathfrak{m}$ has finite global dimension. Hence each $R_\mathfrak{m}$ is a regular local ring by the Auslander--Buchsbaum--Serre theorem \cite[5.84]{lamlect}, and so $R$ is a regular ring by the global Auslander--Buchsbaum--Serre theorem again.
	\end{proof}
\begin{rmk}\label{nagatarmk}
	It is not true that a commutative noetherian regular ring must have finite global dimension. Indeed, Nagata's example\footnote{The author learned about this example from Birge Huisgen-Zimmermann via Bernhard Keller.} (see \cite{nagata} or \cite[5.96]{lamlect}) provides a counterexample: there exists a commutative regular noetherian domain $R$, of infinite Krull dimension, whose localisations $R_\mathfrak{m}$ at every maximal ideal are regular local rings of finite, but arbitrarily large, Krull dimension (and hence global dimension). Because resolutions localise, $R$ cannot have finite global dimension. The issue is essentially that although $D_\mathrm{sg}(R)$ vanishes, it may not vanish in a `uniform' way.
	\end{rmk}
\begin{cor}
	Let $R$ be a commutative noetherian ring of finite Krull dimension. The following are equivalent:
\begin{itemize}
	\item $R$ is regular.
	\item $R$ has finite global dimension.
	\item $D_\mathrm{sg}R$ vanishes.
	\end{itemize}	
\end{cor}

	For technical reasons, we will need to know that various singularity categories we use are idempotent complete; we recall the notion below.
	\begin{defn}
		Let $\mathcal T$ be a triangulated category. A \textbf{projector} in $\mathcal T$ is a morphism\linebreak $\pi:X \to X$ in $\mathcal T$ with $\pi^2=\pi$. Say that a projector $\pi:X \to X$ \textbf{splits} if $X$ admits a direct summand $X'$ such that $\pi$ is the composition $X \to X' \to X$. Say that $\mathcal T$ is \textbf{idempotent complete} if every projector in $\mathcal T$ splits.
	\end{defn}
	\begin{prop}[\cite{bsidempotent}]
		If $\mathcal T$ is a triangulated category, then there exists an idempotent complete triangulated category ${\mathcal T}^\omega$ and a fully faithful triangle functor ${\mathcal T}\to{\mathcal T}^\omega$, universal among functors from ${\mathcal T}$ into idempotent complete triangulated categories. Call ${\mathcal T}^\omega$ the \textbf{idempotent completion} or the \textbf{Karoubi envelope} of ${\mathcal T}$. The assignment ${\mathcal T}\mapsto{\mathcal T}^\omega$ is functorial.
	\end{prop}

	\begin{prop}[{\cite[5.5]{kalckyang2}}]\label{sgidemref}
		Let $R$ be a Gorenstein ring. If $R$ is a finitely generated module over a commutative complete local noetherian $k$-algebra, then $D_\mathrm{sg}(R)$ is idempotent complete.
	\end{prop}
	\begin{rmk}
		The second condition is satisfied for example when $R$ is a finite-dimensional $k$-algebra, or when $R$ is itself a commutative complete local noetherian $k$-algebra.
	\end{rmk}
	We now introduce a relative version of the singularity category. Let $A$ be a $k$-algebra and let $e\in A$ be an idempotent. Write $R$ for the cornering $eAe$. Note that by \ref{recoll}, the functor $j_!= -\lot_{R} eA$ embeds $D(R)$ into $D(A)$. In fact, since $D(R)=\langle R\rangle$, we have $j_{!}D(R)=\langle eA \rangle$. Similarly, restricting to compact objects shows that $j_{!}\per R = \thick (eA) \subseteq \per A$. 
	\begin{defn}[Burban--Kalck \cite{burbankalck}, Kalck--Yang \cite{kalckyang}]\label{relsingcatdefn}
		Let $A$ be an algebra over $k$, and let $e\in A$ be an idempotent. Write $R$ for the cornering $eAe$. The \textbf{relative singularity category} is the Verdier quotient $$\Delta_R(A) \coloneqq  \frac{D^b(A)}{j_! \per R} \cong \frac{D^b(A)}{\thick (eA)}.$$
	\end{defn}
	In \cite{kalckyang2}, this is referred to as the \textbf{singularity category of $A$ relative to $e$}. We immediately turn to dg singularity categories:
	\begin{defn}\label{dgsing}
		Let $A$ be a $k$-algebra. The \textbf{dg singularity category} of $A$ is the Drinfeld quotient ${D}^{\mathrm{dg}}_\mathrm{sg}(A)\coloneqq D^b_{\mathrm{dg}}(A) / \cat{per}_\mathrm{dg}(A)$. If $e\in A$ is an idempotent, write $R$ for the cornering $eAe$ and $j_!$ for the functor $-\otimes_R^\mathbb{L}eA:D(R)\to D(A)$. It is easy to see that $j_!$ admits a dg enhancement. The \textbf{dg relative singularity category} is the Drinfeld quotient
		$$\Delta^{\mathrm{dg}}_R(A) \coloneqq  \frac{D_{\mathrm{dg}}^b(A)}{j_! \per_{\mathrm{dg}} R} \cong \frac{D_{\mathrm{dg}}^b(A)}{\thick (eA)}.$$
	\end{defn}
	By \ref{drinfeldpretr}, we have  $[{D}^{\mathrm{dg}}_\mathrm{sg}(A)]\cong D_\mathrm{sg}(A)$ and $[{\Delta}^{\mathrm{dg}}_R(A)]\cong {\Delta}_R(A)$. The following easy lemma is useful:
	
	\begin{lem}\label{drinfeldlem}The objects of ${D}^{\mathrm{dg}}_\mathrm{sg}(A)$ are precisely those ind-objects $X \in \cat{ind}D^b_{\mathrm{dg}}(A)$ such that $\varinjlim X$ is acyclic and there is an $M \in D^b_{\mathrm{dg}}(A)$ with a map $M \to X$ with ind-perfect cone.
	\end{lem}
	\begin{proof}
		By the definition of the Drinfeld quotient, the objects of ${D}^{\mathrm{dg}}_\mathrm{sg}(A)$ are precisely those ind-objects $X \in \cat{ind}D^b_{\mathrm{dg}}(A)$ such that:
		\begin{itemize}
			\item $\dgh_{\cat{ind}{D}_{\mathrm{dg}}^b}(\per_{\mathrm{dg}}(A),X)$ is acyclic.
			\item There exists $M \in D^b_{\mathrm{dg}}(A)$ and a map $f:M \to X$ with $\mathrm{cone}(f)\in \cat{ind}(\per_{\mathrm{dg}}(A))$.
		\end{itemize}
It will hence be enough to prove that $\varinjlim X$ is acyclic if and only if $\dgh_{\cat{ind}{D}_{\mathrm{dg}}^b}(\per_{\mathrm{dg}}(A),X)$ is acyclic. Let $P \in \per_{\mathrm{dg}}(A)$. Then we have quasi-isomorphisms
\begin{align*}
\dgh_{\cat{ind}D^b_\mathrm{dg}(A)}(P,X)  \coloneqq & \varinjlim \dgh_{D^b_\mathrm{dg}(A)}(P,X) && \text{by definition}\\
 \cong &\varinjlim\dgh_{D_\mathrm{dg}(A)}(P,X) && \text{because } D^b_\mathrm{dg}(A)\text{ embeds in } D_\mathrm{dg}(A)\\
 \cong &\dgh_{D_\mathrm{dg}(A)}(P,\varinjlim X) && \text{because perfect objects are compact}\\
 \simeq & \R\hom_A(P,\varinjlim X) && \text{because } \dgh \text{ is a model for } \R\hom.
\end{align*}
Note that it was necessary to embed $D^b_\mathrm{dg}(A)$ into $ D_\mathrm{dg}(A)$ as $\varinjlim X$ need not have bounded cohomology. If $\varinjlim X$ is acyclic then so is $\R\hom_A(P,\varinjlim X)$ for any $P$ at all, and hence if $P$ is in addition perfect then $\dgh_{\cat{ind}D^b_\mathrm{dg}(A)}(P,X)$ is acyclic. Conversely if $\dgh_{\cat{ind}D^b_\mathrm{dg}(A)}(P,X)$ is acyclic for every perfect $P$, then in particular taking $P=A$ we see that $$\dgh_{\cat{ind}D^b_\mathrm{dg}(A)}(A,X)\simeq \R\hom_A(A,\varinjlim X) \simeq \varinjlim X$$ is acyclic.
	\end{proof}
\section{Gorenstein rings and the stable category}\label{stabcatsec}
When $R$ is a Gorenstein ring, Buchweitz \cite{buchweitz} noticed that the singularity category of $R$ has an alternate interpretation as the stable category of maximal Cohen-Macaulay $R$-modules. In this section we review Buchweitz's theorem and do some computations with stable Ext modules.
\subsection{Hypersurface singularities}
For completeness, we recall some basic definitions from singularity theory.
\begin{defn}
	A complete local \textbf{hypersurface singularity} is a complete local ring of the form $k\llbracket x_1,\ldots, x_n \rrbracket / \sigma$, where $\sigma$ is a nonzero element of the maximal ideal $\mathfrak{m}_{k\llbracket x_1,\ldots, x_n \rrbracket}$.
	\end{defn}
\begin{defn}
	Let $\sigma \in \mathfrak{m}_{k\llbracket x_1,\ldots, x_n \rrbracket}$ be nonzero. The \textbf{Jacobian ideal} (or the \textbf{Milnor ideal}) $J_\sigma$ of $\sigma$ is the ideal of $k\llbracket x_1,\ldots, x_n \rrbracket$ generated by the partial derivatives $\frac{\partial \sigma}{\partial x_i}$ for $i=1,\ldots,n$. The \textbf{Milnor algebra} $M_\sigma$ is the algebra $k\llbracket x_1,\ldots, x_n \rrbracket/J_\sigma$ and the \textbf{Milnor number} $\mu_\sigma$ is the dimension (over $k$) of the Milnor algebra. The \textbf{Tjurina algebra} of $\sigma$ is the quotient $$T_\sigma\coloneqq \frac{k\llbracket x_1,\ldots, x_n \rrbracket}{(\sigma, J_\sigma)}.$$The \textbf{Tjurina number} $\tau_\sigma$ is the dimension of the Tjurina algebra.
\end{defn}
\begin{defn}
	Say that the hypersurface singularity $ k\llbracket x_1,\ldots, x_n \rrbracket/\sigma$ is \textbf{isolated} if the Milnor number $\mu_\sigma$ is finite.
\end{defn}
\begin{rmk}
	In the holomorphic setting, this is equivalent to the usual definition: the singular locus is an isolated point \cite[2.3]{singsintro}.
\end{rmk}

\subsection{Gorenstein rings}
We recall some standard facts about Gorenstein rings.
\begin{defn}
	Let $R$ be a noncommutative two-sided noetherian ring. Say that $R$ is \textbf{Gorenstein} (or \textbf{Iwanaga-Gorenstein}) if it has finite injective dimension over itself as both a left module and a right module.
\end{defn}
\begin{rmk}
	In this setting, the left injective dimension of $R$ must necessarily agree with the right injective dimension \cite[Lemma A]{zaks}. In general, $R$ might have infinite injective dimension over itself on one side and finite injective dimension on the other.
\end{rmk}
Local complete intersections are Gorenstein:
\begin{prop}[e.g.\ {\cite[21.19]{eisenbud}}]\label{hypsgor}
	Let $S$ be a commutative noetherian regular local ring and $I \subseteq S$ an ideal generated by a regular sequence. Then $R\coloneqq S/I$ is Gorenstein.
\end{prop}
In particular, complete local hypersurface singularities are Gorenstein.

\subsection{The stable category}
When $R$ is a Gorenstein ring, then its singularity category has a more algebraic interpretation as a certain stable category of modules. In this section, we follow Buchweitz's seminal unpublished manuscript \cite{buchweitz}.	
\begin{defn}
	Let $R$ be a Gorenstein ring. If $M$ is an $R$-module, write $M^\vee$ for the $R$-linear dual $\hom_R(M,R)$. A finitely generated $R$-module $M$ is \textbf{maximal Cohen--Macaulay} or just \textbf{MCM} if the natural map $\R\hom_R(M,R)\to M^\vee$ is a quasi-isomorphism.
\end{defn}
\begin{rmk}
	An equivalent characterisation of MCM $R$-modules is that they are those modules $M$ for which $\ext_R^j(M,R)$ vanishes whenever $j> 0$.
\end{rmk}
Let $R$ be a Gorenstein ring and $M,N$ be two MCM $R$-modules. Say that a pair of maps $f,g:M\to N$ are \textbf{stably equivalent} if their difference $f-g$ factors through a projective module. Stable equivalence is an equivalence relation, and we denote the set of stable equivalence classes of maps $M \to N$ by $\underline{\hom}_R(M,N)$. We refer to such an equivalence class as a \textbf{stable map}. The \textbf{stable category} of $R$-modules is the category $\stab R$ whose objects are the MCM $R$-modules and whose morphisms are the stable maps. Composition is inherited from $\cat{mod}\text{-}R$.
\begin{defn}\label{weaksyz}
	Let $R$ be a Gorenstein ring. For each $R$-module $X$, choose a surjection $f:R^n \onto X$ and set $\Omega X\coloneqq \ker f$. We refer to $\Omega$ as a \textbf{syzygy} of $X$.
\end{defn}
\begin{prop}
	The assignment $X \mapsto \Omega X$ is a well-defined endofunctor of $\stab R$.
\end{prop}
\begin{rmk}
	In particular, the ambiguities in the definition of syzygies are resolved upon passing to the stable category: syzygies are really only defined up to projective modules, but projective modules go to zero in $ \stab R$. Moreover, the syzygy of a MCM module is again MCM; this is not hard to see by continuing $R^n\onto X$ to a free resolution $F$ of $X$, truncating and shifting to get a resolution $(\tau_{\leq -1}F)[-1]$ of $\Omega X$, dualising, and using that $F^\vee$ has cohomology only in degree zero to see that $(\tau_{\leq -1}F)^\vee[1] \simeq \R\hom_R(\Omega X, R)$ has cohomology only in degree zero.
\end{rmk}
\begin{prop}
	Let $R$ be a Gorenstein ring. Then the syzygy functor $\Omega$ is an autoequivalence of $\stab R$. Its inverse $\Omega^{-1}$ makes $\stab R$ into a triangulated category.
\end{prop}
\begin{rmk}
	Later on, we will primarily be interested in situations where $\Omega\cong \Omega^{-1}$, so the use of the inverse syzygy functor instead of the syzygy functor itself is unimportant to us.
\end{rmk}
A famous theorem of Buchweitz tells us that the stable category we have just defined is the same as the singularity category:
\begin{thm}\label{mcmisdsg}Let $R$ be a Gorenstein $k$-algebra. The categories $D_{\mathrm{sg}}(R)$ and $\stab R$ are triangle equivalent, via the map that sends a MCM module $M$ to the object $M \in D^b(R)$.
\end{thm}
\subsection{The dg stable category}
Let $R$ be a Gorenstein ring. We regard the dg singularity category $D^\mathrm{dg}_{\mathrm{sg}}(R)$ as a dg enhancement of $\stab R$. 
\begin{defn}
	Let $R$ be a Gorenstein $k$-algebra and let $M,N$ be elements of $D^\mathrm{dg}_{\mathrm{sg}}(R)$. Write $\R\underline{\hom}_R(M,N)$ for the complex $\dgh_{D^\mathrm{dg}_{\mathrm{sg}}(R)}(M,N)$ and write $\R\underline{\enn}_R(M)$ for the dga $\dge_{D^\mathrm{dg}_{\mathrm{sg}}(R)}(M)$.
\end{defn}We denote Ext groups in the singularity category by $\underline{\ext}$. Note that $\underline\hom$ coincides with $\underline\ext^0$, and that $\underline{\ext}^j(M,N)\cong H^j\R\underline{\hom}_R(M,N)$. In order to investigate the stable Ext groups, we recall the notion of the complete resolution of a MCM $R$-module -- the construction works for arbitrary complexes in $D^b(\cat{mod}\text{-}R)$.
\begin{defn}[{\cite[5.6.2]{buchweitz}}]
	Let $R$ be a Gorenstein $k$-algebra and let $M$ be any MCM $R$-module. Let $P$ be a projective resolution of $M$, and let $Q$ be a projective resolution of $M^\vee$. Dualising and using that $(-)^\vee$ is an exact functor on MCM modules and on projectives gives us a projective coresolution $M \to Q^\vee$. The \textbf{complete resolution} of $M$ is the (acyclic) complex $\mathbf{CR}(M)\coloneqq \mathrm{cocone}(P \to Q^\vee)$. So in nonpositive degrees, $\mathbf{CR}(M)$ agrees with $P$, and in positive degrees, $\mathbf{CR}(M)$ agrees with $Q^\vee[-1]$.
\end{defn}
\begin{prop}[{\cite[6.1.2.ii]{buchweitz}}]\label{buchcohom}Let $R$ be a Gorenstein $k$-algebra and let $M,N$ be MCM $R$-modules. Then $$\underline{\ext}_R^j(M,N) \cong H^j\dgh_R(\mathbf{CR}(M),N) .$$
\end{prop}
\begin{cor}\label{stabextisext}
	Let $R$ be a Gorenstein $k$-algebra and let $M,N$ be MCM $R$-modules.\begin{enumerate}
		\item If $j>0$ then $\underline{\ext}_R^j(M,N)\cong\ext_R^j(M,N)$. 
		\item If $j<-1$ then $\underline{\ext}_R^j(M,N)\cong\tor^R_{-j-1}(N,M^\vee)$.
	\end{enumerate}
\end{cor}
\begin{proof}
	Let $P\to M$ and $Q \to M^\vee$ be projective resolutions. If $j>0$ we have $$H^j\hom_R(\mathbf{CR}(M),N)\cong H^j\dgh_R(P,N)\cong\ext^j_R(M,N)$$whereas if $j<-1$ we have
	$$H^j\dgh_R(\mathbf{CR}(M),N)\cong H^j\dgh_R(Q^\vee[-1],N)\cong H^j(N\lot_R M^\vee[1]) \cong \tor^R_{-j-1}(N,M^\vee)$$ where we use \cite[6.2.1.ii]{buchweitz} for the quasi-isomorphism $\R\hom_R(Q^\vee,N) \simeq N\lot_R M^\vee$.
\end{proof}
Finally, we recall AR duality, which will assist us in some computations later:
\begin{prop}[Auslander--Reiten duality \cite{auslander}]\label{arduality}
	Let $R$ be a commutative complete local	Gorenstein isolated singularity of Krull dimension $d$. Let $M,N$ be MCM $R$-modules. Then we have $$\underline{\hom}_R(M,N) \cong {\ext}_R^1(N,\Omega^{2-d} M)^*$$where the notation $(-)^*$ denotes the $k$-linear dual.
\end{prop}

	\section{Partial resolutions and the singularity functor}
	We introduce the class of `noncommutative partial resolutions' as a key example of rings with idempotents; these will be our main objects of study later. Given a ring $A$ with an idempotent $e$ and cornering $R=eAe$, we can construct two objects: a dga $\dq$, and a triangulated or dg category $\Delta_R(A)$. We link these two constructions by recalling some results of Kalck and Yang \cite{kalckyang,kalckyang2} on relative singularity categories, as seen from the perspective of the derived quotient. We define a functor $\Sigma: \per(\dq) \to D_\mathrm{sg}(R)$, the \textbf{singularity functor}, and enhance it to a dg functor. We prove our key technical theorem stating that when $A\cong \enn_R(R\oplus M)$ is a noncommutative partial resolution of $R$, the singularity functor induces a quasi-isomorphism $\dq \xrightarrow{\simeq} \tau_{\leq 0}\R\underline{\enn}_R(M)$.
	
	\subsection{Noncommutative partial resolutions}
	We introduce a useful class of examples of rings with idempotents. These rings will be our main objects of study.
	\begin{defn}\label{partrsln}
		Let $R$ be a commutative Gorenstein $k$-algebra. A $k$-algebra $A$ is a \textbf{(noncommutative) partial resolution} of $R$ if it is of the form $A\cong\enn_R(R\oplus M)$ for some MCM $R$-module $M$. Note that $A$ is a finitely generated module over $R$, and hence itself a noetherian $k$-algebra. Say that a partial resolution is a \textbf{resolution} if it has finite global dimension.
	\end{defn}
Clearly the data of a noncommutative partial resolution of $R$ is the same thing as the data of a MCM $R$-module.
\begin{rmk}\label{premcmrmk}
If $A=\enn_R(R\oplus M)$ is a NCCR of $R$ with $M$ maximal Cohen--Macaulay, then $A$ is automatically a noncommutative partial resolution of $R$ in our sense. In \cite{vdbnccr} it is remarked that in all examples of NCCRs given, one may indeed take $M$ to be MCM. We are essentially removing smoothness (i.e. $A$ has finite global dimension) and crepancy ($A$ is a MCM $R-$module) at the cost of assuming an extra mild hypothesis on $M$. One can develop most of the results of this section without the assumption that $M$ is MCM, but one can of course no longer make arguments using the stable category. It seems that all one needs for our key technical theorem (\ref{qisolem}) to go through is reflexivity of $M$ along with the hypothesis that $\ext^1_R(M,R)$ vanishes (see \ref{mcmrmk}). In our main application to threefold flops, in fact this is equivalent to $M$ being MCM: reflexivity implies that $M$ has depth at least 2, and now local duality implies that if $\ext^1_R(M,R)$ vanishes then $M$ has depth 3; i.e.\ is MCM.
	\end{rmk}
	Recall that if $M$ is an $R$-module then we write $M^\vee$ for the $R$-linear dual $\hom_R(M,R)$. If $A=\enn_R(R\oplus M)$ is a noncommutative partial resolution of $R$, observe that $e\coloneqq \id_R$ is an idempotent in $A$. One has $eAe\cong R$, $Ae \cong R \oplus M$, and $eA \cong R \oplus M^\vee$; in particular $(Ae)^\vee\cong eA$. Hence we have a ring isomorphism $A=\enn_R(R\oplus M)\cong \enn_{eAe}(Ae)$ which will be useful to us (note that not every stratifying idempotent yields such an isomorphism). Observe that $Ae\cong M $ in the singularity category, and indeed we have $A/AeA\cong \underline{\enn}(M)$. Note that, to the above data, one can canonically attach a dga $\dq$.
	\begin{defn}
		Let $R$ be a Gorenstein $k$-algebra and $(A,e)$ a noncommutative partial resolution. We refer to the derived quotient $\dq$ as the \textbf{derived exceptional locus}.
		\end{defn}
	\begin{rmk}\label{delrmk}
		The name `derived exceptional locus' for $\dq$ is motivated by the recollement of \ref{recoll}. Loosely, one can think of $A$ as some noncommutative scheme $X$ over $\spec R$, and then \ref{recoll} identifies the kernel of the derived pushforward functor $D(X) \to D(\spec R)$ with the derived category of $\dq$. When $R \to A$ is actually derived equivalent to some partial resolution $X \xrightarrow{\pi} \spec R$ then $D(\dq)$ is equivalent to the category $\ker \R\pi_* \subseteq D(X)$, which can be thought of as a `derived thickening' of the exceptional locus $E$ (indeed, the cohomology sheaves of $\dq$ are all set-theoretically supported on $E$). 
		\end{rmk}
	 Since ${R}\cong 0$ in the stable category, $\R \underline{\enn}_{R}({R}\oplus M)$ is naturally quasi-isomorphic to $\R \underline{\enn}_{R}(M)$. Hence, the stable derived endomorphism algebra $\R \underline{\enn}_R(M)$ gets the structure of an $A$-module. Clearly, $\R \underline{\enn}_R(M) e$ is acyclic, and so $\R \underline{\enn}_R(M)$ is in fact a module over $\dq$.
	 
	 In the sequel, we will frequently refer to the following setup:
	 \begin{setup}\label{sfsetup}
	 	Let $R$ be a commutative complete local Gorenstein $k$-algebra. Fix a MCM $R$-module $M$ and let $A=\enn_R(R\oplus M)$ be the associated partial resolution. Let $e\in A$ be the idempotent $e=\id_R$. 
	 \end{setup}
 The condition that $R$ is complete local is merely there to ensure that $D_{\mathrm{sg}}(R)$ is idempotent complete (by \ref{sgidemref}); this will be necessary for some our technical arguments about singularity categories to work. The author is not aware of any weaker easily verifiable condition to ensure that a commutative Gorenstein $k$-algebra has idempotent complete singularity category. 
	\begin{rmk}
		We could also stipulate that $M$ is not projective (i.e.\ nonzero in $\stab R$) in the definition of a noncommutative partial resolution. This is a harmless assumption as partial resolutions with $M$ projective are uninteresting: in view of \ref{qisolem} any such partial resolution has $\dq\simeq 0$. However our arguments do not require this assumption, and we find it more illuminating to allow for the possibility that $M\cong 0$ in the stable category. The only time we need to assume that $M$ is not projective is right at the end in the proof of \ref{recov} (which is clearly false without this assumption). 
	\end{rmk}
	
	\begin{rmk}
		In addition to the above conditions, suppose that $R$ is of finite MCM-representation type (e.g.\ an ADE hypersurface), and that $A$ is the Auslander algebra of $R$ (i.e.\ take $M$ to be the sum of all the indecomposable non-projective MCM modules). Let $S$ denote the quotient of $A/AeA$ by its radical. If $S$ is one-dimensional over $k$ (i.e.\ $M$ was actually indecomposable), then it follows that the derived exceptional locus $\dq$ is quasi-isomorphic to the dg Auslander algebra $\Lambda_{dg}(\stab R)$ of Kalck and Yang \cite{kalckyang}. Indeed, in this setting they are both quasi-isomorphic to the Koszul dual of the augmented algebra $\R\enn_A(S)$; this holds for the derived quotient by results of the author's thesis \cite{me} while it holds for the dg Auslander algebra by the proof of \cite[5.5]{kalckyang}. The statement remains true if one drops the one-dimensional hypothesis on $S$; this follows from the results of \cite{ddefpt}. The dga $\Lambda_{dg}(\stab R)$ together with the action of the AR-translation on $\stab R$ determines the relative singularity category of $A$ as a dg category, and taking the dg quotient this determines $\stab R$ as a dg category. It would be interesting to compare this with our approach to determining $\stab R$ from $\dq$. Loosely, knowledge of the AR-translation should be equivalent to knowledge of the periodicity element $\eta$ of \S7.
		\end{rmk}
	
		\subsection{The singularity functor}Let $A$ be a right noetherian $k$-algebra with an idempotent $e$, and write $R\coloneqq eAe$ for the cornering. Recall from \ref{recoll} the existence of the recollement $D(\dq)\recol D(A) \recol D(R)$, and recall from \ref{relsingcatdefn} the definition of the relative singularity category $\Delta_R(A)\coloneqq D^b(A)/\thick(eA)$.  The map $j^*: D(A) \to D(R)$ sends $\thick(eA)$ into $\per R$, and hence defines a map $j^*:\Delta_R(A) \to D_\mathrm{sg}R$. In fact, $j^*$ is onto, which follows from \cite[3.3]{kalckyang}. We are about to identify its kernel.
	\begin{defn}\label{dfgdefn}
		Write $ D_\mathrm{fg}(\dq)$ for the subcategory of $D(\dq)$ on those modules whose total cohomology is finitely generated over $A/AeA$. Similarly, write $\per_\mathrm{fg}(\dq)$ for the subcategory of $\per(\dq)$ on those modules whose total cohomology is finitely generated over $A/AeA$.
	\end{defn}
	\begin{lem}\label{kerjlem}
		The kernel of the map $j^*:\Delta_R(A) \to D_\mathrm{sg}R$ is precisely $D_\mathrm{fg}(\dq)$.
	\end{lem}
	
	\begin{proof}As in the proof of \cite[6.13]{kalckyang}, a result on idempotents \cite[3.3]{kalckyang} combined with a proposition of Verdier \cite[II.2.3.3]{verdierthesis} shows that $\ker j^* \cong \thick_{D(A)}(\cat{mod}\text{-}A/AeA)$, with the isomorphism given by the projection map. Hence it suffices to show that we have an isomorphism $\thick_{D(A)}(\cat{mod}\text{-}A/AeA)\cong D_\mathrm{fg}(\dq)$. But this can be shown to hold via a modification of the proof of \cite[2.12]{kalckyang}.
	\end{proof}
	\begin{rmk}\label{persrmk}
		If $A/AeA$ is a finite-dimensional algebra, let $\mathcal{S}$ be the set of one-dimensional $A/AeA$-modules corresponding to a set of primitive orthogonal idempotents for $A/AeA$. Then $D_\mathrm{fg}(\dq)\cong \thick(\mathcal{S})$. Because each simple in $\mathcal{S}$ need not be perfect over $\dq$, the category $\per_\mathrm{fg}(\dq)$ may be smaller than $D_\mathrm{fg}(\dq)$. If each simple is perfect over $A$, or if $\dq$ is homologically smooth, then we have an equivalence $D_\mathrm{fg}(\dq)\cong \per_\mathrm{fg}(\dq)$. 
	\end{rmk}
	When the singularity category is idempotent complete, Kalck and Yang observed that there is a triangle functor ${\Sigma: \per(\dq) \to D_{\mathrm{sg}}(R)}$, sending $\dq$ to the right $R$-module $Ae$. We establish this with a series of results. Recall that when $\mathcal T$ is a triangulated category, ${\mathcal T}^\omega$ denotes the idempotent completion of $\mathcal T$.
	\begin{lem}\label{flem}
		There is a triangle functor $F:\per(\dq)\to \Delta_R(A)^\omega$ which sends $\dq$ to the object $A$.
	\end{lem}
	\begin{proof}
		As in \cite[2.12]{kalckyang} (which is an application of Neeman--Thomason--Trobaugh--Yao localisation; cf.\ \ref{ntty}), the map $i^*$ gives a triangle equivalence $$i^*:\left(\frac{\per A}{j_!\per R}\right)^\omega \xrightarrow{\cong} \per (\dq).$$The inclusion $\per A \into D^b(A)$ gives a map $G:\per A / j_!\per R \to \Delta_R(A)$, which is a triangle equivalence if $A$ has finite right global dimension. The composition $$F:\per(\dq) \xrightarrow{(i^*)^{-1}} \left(\frac{\per A}{j_!\per R}\right)^\omega \xrightarrow{G^\omega} \Delta_R(A)^\omega$$ is easily seen to send $\dq$ to $A$. 
	\end{proof}
	\begin{lem}\label{preontolem}
		Suppose that $A$ is of finite right global dimension. Then the map $F$ of \ref{flem} is a triangle equivalence $\per(\dq)\to \Delta_R(A)^\omega$.
	\end{lem}	 
	\begin{proof}
		When $A$ has finite global dimension then $F$ is a composition of triangle equivalences and hence a triangle equivalence.
	\end{proof}
	
	\begin{prop}[cf.\ {\cite[6.6]{kalckyang2}}]\label{kymap}
		Suppose that $D_\mathrm{sg}(R)$ is idempotent complete. Then there is a map of triangulated categories $\Sigma:\per(\dq) \to D_\mathrm{sg}(R)$, sending $\dq$ to $Ae$. Moreover $\Sigma$ has image $\thick_{D_\mathrm{sg}(R)}(Ae)$.
	\end{prop}
	\begin{proof}		
		We already have a map $j^*:\Delta_R(A)\to D_\mathrm{sg}(R)$, with kernel $D_\mathrm{fg}(\dq)$. Let $\Sigma$ be the composition
		$$\Sigma: \per(\dq) \xrightarrow{F} \Delta_R(A)^\omega \xrightarrow{(j^*)^\omega} D_{\mathrm{sg}}(R)$$ of the functor $F$ of \ref{flem} and the idempotent completion of $j^*$. It is easy to see that $\Sigma$ sends $\dq$ to $Ae$, and since $A$ generates $\per A / j_!\per R$, then $\Sigma $ has image $\thick(Ae)$.
	\end{proof}
	For future reference, it will be convenient to give $\Sigma$ a name.
	\begin{defn}
		We refer to the triangle functor $\Sigma$ of \ref{kymap} as the \textbf{singularity functor}.
	\end{defn}
Occasionally, it will be useful for us to have a classification of all thick subcategories of a singularity category. For isolated abstract hypersurfaces, thick subcategories were classified by Takahashi \cite{takahashi} in terms of the geometry of the singular locus\footnote{The author learnt about this result from Martin Kalck.}. We recall a very special case of this classification:
\begin{thm}[{\cite[6.9]{takahashi}}]\label{takthm}
	Let $R$ be a complete local isolated hypersurface singularity. The only nonempty thick subcategories of $D_\mathrm{sg}(R)$ are the zero subcategory and all of $D_\mathrm{sg}(R)$. In particular, if $M \in D^b(R)$ then $M$ generates the singularity category unless $M$ is perfect, in which case $\thick_{D_\mathrm{sg}(R)}(M)\cong 0 $.
	\end{thm}

	\begin{prop}\label{kymapbetter}
		Suppose that $D_\mathrm{sg}(R)$ is idempotent complete. Then the triangle functor $\Sigma$ induces an essentially surjective triangle functor $${\bar{\Sigma}}:\frac{\per(\dq)}{\ker\Sigma} \to \thick_{D_\mathrm{sg}(R)}(Ae).$$
	\end{prop}
	\begin{proof}
		Follows immediately from \ref{kymap}. 
	\end{proof}
We note that by \cite[2.1.10]{neemantxt}, the kernel of the quotient functor $\per{\dq} \to \frac{\per(\dq)}{\ker\Sigma}$ is precisely $\thick(\ker\Sigma) = (\ker\Sigma)^\omega$.
	\p When $A$ is smooth, one can do better:
	\begin{lem}\label{ontolem}
		Suppose that $A$ is of finite right global dimension and $D_\mathrm{sg}(R)$ is idempotent complete. Then the singularity functor $\Sigma$ is onto.
	\end{lem}
	\begin{proof}
		The singularity functor is the equivalence of \ref{preontolem} followed by the surjective triangle functor $j^*:\Delta_R(A)\to D_\mathrm{sg}(R)$.
	\end{proof}

\begin{rmk}
	The converse of the above statement is not true: it may be the case that $A$ has infinite global dimension but $\Sigma$ is still onto (in view of Takahashi's theorem this is perhaps not surprising). One general class of counterexamples, pointed out to the author by Martin Kalck, is the following. Let $R$ be an even-dimensional ADE singularity, let $M$ be an indecomposable non-projective MCM module and let $A$ be the associated partial resolution. Then by \cite{kiwy}, the ring $A$ does not have finite global dimension unless $M$ is the only indecomposable non-projective MCM module. But by Takahashi's theorem \ref{takthm}, $M$ always generates the singularity category. A concrete counterexample in this spirit, found by Michael Wemyss, is the following. Let $R$ be the complete local hypersurface in $\A^4$ defined by the equation $uv = y(x^2+y^3)$. Let $M$ be the $R$-module $M=(u,x^2+y^3)$ and let $A$ be the associated partial resolution. Kn{\"{o}}rrer periodicity \cite{knoerrer} gives an equivalence between the singularity category of $R$ and the stable category of modules on the $D_5$ curve singularity $x^2y + y^4=0$, which has finite CM type. Let $N$ be the module over the $D_5$ curve corresponding to $M$. Then $A$ has infinite global dimension, because $N$ is maximal rigid but not cluster tilting \cite[2.5]{bikr}. One can check by hand using Auslander--Reiten theory that $N$ generates the singularity category. Note that the AR quiver and the AR sequences appear in \cite[9.11]{yoshino}, where $N$ is labelled by either $A$ or $B$ (the two are equivalent up to quiver automorphism).
	\end{rmk}

 	\begin{prop}[{\cite[1.2]{kalckyang2}}]\label{dsgsmooth}Suppose that all finitely generated $A/AeA$-modules have finite projective dimension over $A$. Suppose also that $D_\mathrm{sg}(R)$ is idempotent complete. Then the singularity functor induces a triangle equivalence $${\bar{\Sigma}}:\frac{\per(\dq)}{D_{\mathrm{fg}}(\dq)} \to D_\mathrm{sg}(R).$$
 \end{prop}

 \begin{proof}
 	The proof of \cite[1.2]{kalckyang2} shows that the quotient $\frac{\per A}{j_!\per R}$ is idempotent complete. Hence we may apply \ref{kerscor} to deduce that the kernel of $\Sigma$ is precisely the thick subcategory $\per_{\mathrm{fg}}(\dq)$. As in \ref{ksigl}, we have $\per_{\mathrm{fg}}(\dq) = D_{\mathrm{fg}}(\dq)$.
 \end{proof}

	\begin{rmk}
	This equivalence is essentially the same as that of \cite[5.1.1]{dtdvvdb}.
\end{rmk}

	\subsection{The dg singularity functor}
	We enhance some of the results of the last section to the dg setting. In this section suppose that $A$ is a right noetherian $k$-algebra with idempotent $e$. Put $R\coloneqq eAe$ and assume that $D_{\mathrm{sg}}(R)$ is idempotent complete. 
	\begin{prop}\label{Fdg}The singularity functor $\Sigma:\cat{per}(\dq) \to D_{\mathrm{sg}}(R)$ lifts to a dg functor $ \Sigma_\mathrm{dg}:\cat{per}_\mathrm{dg}(\dq) \to {D}^{\mathrm{dg}}_\mathrm{sg}(R)$, which we refer to as the \textbf{dg singularity functor}.
	\end{prop}
	\begin{proof}
		We simply mimic the proof (\ref{kymap}) from the triangulated setting. Recalling from \ref{kymap} the construction of $\Sigma$ as a composition $\per(\dq) \xrightarrow{\Sigma_1} \Delta_R(A) \xrightarrow{\Sigma_2} {D}_\mathrm{sg}(R)$, we lift the two maps separately to dg functors. To lift $\Sigma_1$, first note that \ref{ntty} and \ref{Rcoloc} provide a homotopy cofibre sequence of dg categories $\per_{\mathrm{dg}}R \to \per_{\mathrm{dg}}A \to \per_{\mathrm{dg}}(\dq)$, in which the first map is $j_!$. There is a homotopy cofibre sequence $\per_{\mathrm{dg}}R \to D_\mathrm{dg}^b(A) \to \Delta^\mathrm{dg}_R(A)$, and we can extend $\id: \per_{\mathrm{dg}}R \to \per_{\mathrm{dg}}R$ and the inclusion $\per_{\mathrm{dg}}A \into D_\mathrm{dg}^b(A)$ into a morphism of homotopy cofibre sequences, which gives a lift of $\Sigma_1$. Lifting $\Sigma_2=j^*$ is similar and uses the sequence $\per_{\mathrm{dg}}R \to D_\mathrm{dg}^b(R) \to {D}^{\mathrm{dg}}_\mathrm{sg}(R)$.
	\end{proof}
Let $\ker^\mathrm{dg}\Sigma$ be the kernel of the dg functor $\Sigma$. 
\begin{prop}\label{kymapbetterdg}
	There is a quasi-equivalence of dg categories $$\bar{\Sigma}_\mathrm{dg}: \frac{\per_\mathrm{dg}(\dq)}{\ker^\mathrm{dg}\Sigma}\xrightarrow{\simeq} \thick_{D^\mathrm{dg}_\mathrm{sg}(R)}(Ae)$$which enhances the triangle equivalence $\bar\Sigma$ of \ref{kymapbetter}.
	\end{prop}
	\begin{proof}
	By \ref{Fdg} and \ref{kymap} we have a quasi-essential surjection $$\Sigma_{\mathrm{dg}}:\per_{\mathrm{dg}}(\dq) \onto \thick_{D^\mathrm{dg}_\mathrm{sg}(R)}(Ae)$$ which sends $\dq$ to $Ae$. Hence, $\Sigma_{\mathrm{dg}}$ descends to a quasi-equivalence $$\bar{\Sigma}_\mathrm{dg}: \frac{\per_\mathrm{dg}(\dq)}{\per^\mathrm{dg}_\mathrm{fg}(\dq)}\xrightarrow{\simeq} \thick_{D^\mathrm{dg}_\mathrm{sg}(R)}(Ae)$$which by construction enhances $\bar\Sigma$, since the dg quotient enhances the Verdier quotient.
	\end{proof}
\begin{prop}\label{dgtakahashi}
	With setup as above, if $Ae$ generates the singularity category as a triangulated category, then there is a natural quasi-equivalence $\thick_{D^\mathrm{dg}_\mathrm{sg}(R)}(Ae)\simeq {D^\mathrm{dg}_\mathrm{sg}(R)}$.
	\end{prop}
\begin{proof}
The natural inclusion $\thick_{D^\mathrm{dg}_\mathrm{sg}(R)}(Ae)\into {D^\mathrm{dg}_\mathrm{sg}(R)}$ is always quasi-fully faithful, and it is quasi-essentially surjective by assumption.
	\end{proof}
	
	\subsection{The comparison map}
	The components of the singularity functor give us dga maps $\dge(X) \to \dge(\Sigma X)$. We investigate in detail the case $X = \dq$. Observe that $\Sigma(\dq)\simeq Ae$, and moreover we can canonically identify $\dq$ with the endomorphism dga $\dge_{\per_{\mathrm{dg}}(\dq)}(\dq)$.
	\begin{defn}\label{comparisonmap}
		The \textbf{comparison map} $\Xi:\dq \to \dge_{D^\mathrm{dg}_\mathrm{sg}(R)}(Ae)$ is the component of the dg singularity functor $\Sigma_{\mathrm{dg}}$ at the object $\dq \in {\per_{\mathrm{dg}}(\dq)}$.
	\end{defn}
	In other words, the comparison map is the morphism of dgas given by $$\dq \cong \dge_{\dq}(\dq)\xrightarrow{\Sigma} \dge(\Sigma(\dq))\simeq \dge(Ae).$$The main theorem of this section is that when $A$ is a partial resolution of $R$, defined by a MCM module $M$, the comparison map $$\Xi:\dq \to \R\underline{\enn}_R(M)$$ is a `quasi-isomorphism in nonpositive degrees'. The strategy will be as follows. We are going to use the `Drinfeld quotient' model $B$ for $\dq$ that appears in \ref{drinfeld}, and use this to write down an explicit model for $M$ as an object of the Drinfeld quotient $D^\mathrm{dg}_\mathrm{sg}(R)$. This will allow us to calculate an explicit model for $\R\underline{\enn}_R(M)$. At this point we will already be able to see that $\dq$ and the truncation $\tau_{\leq 0}\R\underline{\enn}_R(M)$ are quasi-isomorphic as abstract complexes. We think of the comparison map $\Xi$ as giving an action of $B$ on (our model for) $M$, which we are able to explicitly identify. We are also able to explicitly identify the action of $\tau_{\leq 0}\R\underline{\enn}_R(M)$ on $M$. This allows us to explicitly identify the comparison map and at this point it will be easy to see that it is a quasi-isomorphism.
	
	\p However, before we begin we must prove a technical lemma. Loosely, to write down a representative of the object $M\cong Ae$ of the singularity category, we need to pick an ind-perfect $R$-module $I$ whose  colimit resolves $Ae$. Keeping \ref{drinfeldmodel} in mind, our preferred resolution will be the bar resolution. The na\"ive thing to do is to simply choose $I$ to be the filtered system of all perfect submodules of the bar resolution. However, with this choice we run into problems with homotopy limits. We wish to choose our $I$ such that if $f:U \into V$ is a morphism in $I$, then $\coker(f)$ is degreewise projective (i.e.\ $f$ is a cofibration). The following construction is one such choice.
	
	\begin{defn}\label{induceddefn}
		Suppose that we are in the situation of Setup \ref{sfsetup}. Let $$\mathrm{Bar}(Ae)\coloneqq\cdots \to Ae \otimes_k R \otimes_k R \to Ae\otimes_k R$$ be the bar resolution of the $R$-module $Ae$, with $Ae \otimes_k R^{\otimes(1-n)}$ in degree $-n$. Say that a dg-$R$-submodule $M$ of $\mathrm{Bar}(Ae)$ is \textbf{induced} if
		\begin{itemize}
			\item $M$ is bounded.
			\item For every $n\geq 0$, the module $M^{-n}$ is isomorphic to an $R$-module of the form $V\otimes_k R$, where $V\subseteq Ae \otimes_k R^{\otimes-n}$ is a finite-dimensional $k$-vector space. Moreover, the inclusion $M^{-n} \into \mathrm{Bar}(Ae)^{-n}$ is induced by applying the exact functor $-\otimes_k R$ to the vector space inclusion $V \into Ae \otimes_k R^{\otimes-n}$.
			\end{itemize}
		In particular, an induced module $M$ is bounded, degreewise free and finitely generated, and hence perfect. Say that an inclusion $M \into M'$ of induced modules is \textbf{induced} if degreewise it is of the form $V \otimes_k R \into V' \otimes_k R$, induced by a linear inclusion $V \into V'\subseteq Ae \otimes_k R^{\otimes-n}$. Let $\mathcal{I}$ be the set of all induced submodules of $\mathrm{Bar}(Ae)$, regarded as a diagram of $R$-modules with morphisms the induced inclusions.
		\end{defn}
	We say that an ind-dg-$R$-module $I=\{I_\alpha\}_\alpha \in \cat{ind}(\cat{Mod}\text{-}R)$ is \textbf{ind-perfect} if each $I_\alpha$ is a perfect $R$-module. This is equivalent to specifying that $I$ is an object of the full subcategory $\cat{ind}(\per R)$. The relevance of ind-perfect objects to us is due to their presence in \ref{drinfeldlem}, where they play a part in the Drinfeld quotient model for the dg singularity category.
	
	\begin{lem}\label{ilem}\hfill
\begin{enumerate}
	\item The diagram $\mathcal{I}$ is an ind-perfect $R$-module.
	\item The colimit of $\mathcal{I}$ is $\mathrm{Bar}(Ae)$.
		\item If $f:M \into M'$ is an induced inclusion of induced modules, then $\coker{f}$ is degreewise free and finitely generated over $R$.
	\end{enumerate}
		\end{lem}
	\begin{proof}
		For the first two claims, we will need to first introduce an auxiliary construction. Let $M$ be an induced \textit{graded} submodule of $\mathrm{Bar}(Ae)$; in other words $M$ is a graded submodule of $\mathrm{Bar}(Ae)$ satisfying the two conditions of \ref{induceddefn}. Of course, $M$ may not be a dg submodule as it may not be closed under the differential. Let $n$ be the largest integer such that $M^{-n}\cong V \otimes_k R$ is nonzero. Then $V$ is a finite-dimensional linear subspace of $Ae\otimes_kR^{\otimes n}$. Pick a finite collection $\{v_i= r_0^i\otimes r^i_1\otimes\cdots \otimes r^i_n\}_{1\leq i \leq m}\subseteq Ae\otimes_kR^{\otimes n}$ such that $V$ is contained in the linear span of the $v_i$. Note that $r^i_0 \in Ae$ and $r^i_j \in R$. For each $1\leq i \leq m$, and $0 \leq j < n$, let $u^j_i\in Ae\otimes_kR^{\otimes (n-1)}$ be the element $u^j_i=r_0^i \otimes r^i_1 \otimes \cdots r^i_jr^i_{j+1} \otimes \cdots \otimes r^i_n $. In other words, $u^j_i\otimes 1$ is (up to sign) the $(j+1)$st monomial appearing in the ($n+1$)-element sum $d(v_i \otimes 1)$, where $d$ is the bar differential. Moreover, let $u^{n}_i \in Ae\otimes_kR^{\otimes (n-1)}$ be the element $r^i_0\otimes\cdots \otimes r^i_{n-1}$. In other words, $u^n_i\otimes 1$ is not quite (up to sign) the last term of the sum $d(v_i\otimes 1)$; it has a $1$ in the rightmost position instead of an $r^i_n$. However, the last term of $d(v_i\otimes 1)$ is certainly an element of $ U^n_i\otimes R$, where $U^n_i$ is the one-dimensional vector space spanned by $u^n_i$. Let $U$ be the finite-dimensional linear subspace of $Ae\otimes_kR^{\otimes (n-1)}$ spanned by all of the $u^j_i$ for $1\leq i \leq m$ and $0 \leq j \leq n$. By construction, the image of $M^{-n}$ under the differential $d$ is a $R$-submodule of $U \otimes_k R$. Write $M^{1-n}\cong U'\otimes_k R$ for some $U'\subseteq Ae\otimes_kR^{\otimes (n-1)}$. Regarding $U'$ and $U$ as subspaces of $Ae\otimes_kR^{\otimes (n-1)}$, let $W$ be their sum, which is again a finite-dimensional subspace. Replacing $M^{1-n}$ by $W\otimes_k R$, we may hence assume that $d(M^{-n})\subseteq M^{1-n}$. Continuing inductively, we produce a dg-$R$-module $\langle M \rangle$ with inclusions $M \into \langle M \rangle \into \mathrm{Bar}(Ae)$. It is easy to see that $\langle M \rangle $ is an induced dg submodule of $\mathrm{Bar}(Ae)$, and moreover that the inclusion $M \into \langle M \rangle$ of graded induced modules is induced. Call $\langle M \rangle $ a \textbf{closure} of $M$.
		
			\p For $(1)$, it is clear that $\mathcal{I}$ is a diagram of perfect modules, so we need only show that it is filtered. For this it will suffice to show that if both $M$ and $M'$ are induced modules, then there exists an induced module $N$ together with induced inclusions $M \into N$ and $M' \into N$. For this, write $M^{-n}=V_n\otimes_k R$ and $M'^{-n}=V'_n\otimes_k R$ for all $n\geq 0$. Fixing an $n$, and regarding each $V_n$ and $V'_n$ as linear subspaces of $Ae\otimes_kR^{\otimes n}$, let $W_n$ be their sum, which is again a finite-dimensional subspace. Let $N'$ be the induced graded submodule of $\mathrm{Bar}(Ae)$ with $N'^{-n}=W_n\otimes_k R$. Let $N=\langle N' \rangle$ be a closure. Then it is clear that $N$ is an induced module, and that the obvious inclusions $M \into N$ and $M' \into N$ are induced maps. 
			
			\p Because colimits commute with tensor products, and a vector space is the colimit of its one-dimensional subspaces, to prove (2) it will suffice to show that if $x \in \mathrm{Bar}(Ae)^{-n}$ then there is an induced submodule $M$ such that $x \in M^{-n}$. Write $x=r_0\otimes\cdots \otimes r_{n+1}$. Let $V$ be the one-dimensional subspace of $Ae\otimes_kR^{\otimes n}$ generated by $r_0\otimes\cdots \otimes r_{n}$. It is clear that $M'\coloneqq(V \otimes_k R)[n]$ is an induced graded submodule of $\mathrm{Bar}(Ae)$ containing $x$. Letting $M=\langle M' \rangle$ be a closure, we see that $M$ satisfies the desired condition.
			
\p For $(3)$, let $M \into M'$ be an induced inclusion. Degreewise, this looks like an inclusion $V \into V'$ of finite-dimensional vector spaces, tensored over $k$ with $R$. Because the tensor product is right exact, the cokernel of $M \into M'$ degreewise looks like $\coker(V \into V')\otimes_k R$, which is clearly degreewise free and finitely generated (it is even itself induced).
		\end{proof}
	
	\begin{rmk}\label{indsprmk}
		Of course, the ind-module $\mathcal{I}$ is more than just ind-perfect: it is ind-strictly perfect, in the sense that each level of $\mathcal{I}$ is a bounded complex of finitely generated projective $R$-modules. This will be of relevance to us soon during the proof of \ref{qisolem}.
		\end{rmk}
	
	\begin{rmk}
		If $\mathcal{F}$ denotes the collection of all perfect submodules of $\mathrm{Bar}(Ae)$, we are not claiming the existence of an isomorphism $\mathcal{I} \cong \mathcal{F}$ of ind-perfect $R$-modules (although such an isomorphism presumably exists and is induced by a natural map $\mathcal{I} \to \mathcal{F}$). Rather, we are just claiming that $\mathcal{I}$ is an ind-perfect module whose colimit agrees with the colimit of $\mathcal{F}$.
		\end{rmk}

	\begin{thm}\label{qisolem}Suppose that we are in the situation of Setup \ref{sfsetup}. Then, for all $j\leq 0$, the comparison map $\Xi:\dq \to \R\underline{\enn}_R(M)$ induces isomorphisms on cohomology $$ H^j(\dq) \xrightarrow{\cong} \underline{\ext}_R^j(M,M).$$
	\end{thm}

	\begin{proof}
		We begin by writing down a model for $M$, which we immediately replace by the isomorphic (in the singularity category!) object $Ae$. Letting $\mathcal{I}$ be as in \ref{induceddefn}, we see that $\mathcal{I}$ is ind-perfect by \ref{ilem}(1) and moreover that $\mathrm{Bar}(Ae)\cong \varinjlim \mathcal{I}$ by \ref{ilem}(2). Noting that $j_!=-\lot_R eA$, we hence see that $\mathcal{I}\otimes_R eA$ is an object of $\cat{ind}j_!\per_{\mathrm{dg}} R$. Since tensor products commute with filtered colimits, we have an isomorphism $\varinjlim(\mathcal{I}\otimes_R eA)\cong \mathrm{Bar}(Ae) \otimes_R eA$. Let $B$ be the model for $\dq$ from \ref{drinfeld}. Put $T\coloneqq \varinjlim(\mathcal{I}\otimes_R eA)$ and observe that $T$ is isomorphic to the $A$-bimodule $(\tau_{\leq -1}B)[-1]$ that appears in \ref{drinfeldmodel} and the alternate proof of \ref{dqexact}. Note that $\mathcal{I}\otimes_R eA$ also comes with a multiplication map $\mu$ to $A$ that lifts the multiplication $Ae\otimes_R eA \to A$. Let $C$ be the cone of this map; then $C$ is an ind-bounded module with a map from $A$ whose cone is in $\cat{ind}j_!\per_{\mathrm{dg}} R$. In fact, $\varinjlim C \simeq \dq$ by \ref{dqexact}. Hence, if $P\in j_!\per_{\mathrm{dg}} R=\thick_{\mathrm{dg}}(eA)$, then $\dgh_{\cat{ind}D_{\mathrm{dg}}(A)}(P,C) \simeq\R\hom_A(P,\dq)\simeq 0$, since $P$ is compact, and we have the semi-orthogonal decomposition of \ref{semiorthog}. Hence $C$ is a representative of the $A$-module $\dq$ in the Drinfeld quotient $\Delta^\mathrm{dg}_R(A)$.
		
		\p Note that the dg functor $j^*: \Delta^\mathrm{dg}_R(A) \to {D}^{\mathrm{dg}}_\mathrm{sg}(R)$ is simply multiplication on the right by $e$. Hence, sending $C$ through this map, we obtain an ind-object $Ce\cong\mathrm{cone}(\mathcal{I} \to Ae)$ that represents $\Sigma(\dq)\simeq Ae$ in the dg singularity category $D^\mathrm{dg}_\mathrm{sg}(R)$. As an aside, one can check this directly: since $\mathrm{Bar}(Ae)$ resolves $Ae$, and mapping cones commute with filtered colimits, it is clear that $\varinjlim Ce$ is acyclic, and that $Ce$ admits a map from $Ae \in D^b(R)$ whose cone is the ind-perfect $R$-module $\mathcal{I}$.

		\p Now we will explicitly identify the dga $\R\underline{\enn}_R(Ae)\simeq \dge_{D^\mathrm{dg}_\mathrm{sg}(R)}(Ce)$. Write $Ce=\{W_\alpha\}_\alpha$, where each $W_\alpha$ is a cone $V_\alpha \to Ae$ with $V_\alpha$ induced. For the remainder of this proof, we use $\dgh$ to denote the hom-complexes in the dg derived category of $R$ for brevity. We similarly use $\dge$ to denote endomorphism dgas in the dg derived category of $R$. We compute
		\begin{align*}
		\dge_{D^\mathrm{dg}_\mathrm{sg}(R)}(Ce) & \coloneqq \varprojlim_\alpha\varinjlim_\beta \dgh\left(\mathrm{cone}(V_\alpha \to Ae), W_\beta\right)
		\\ & \cong \varprojlim_\alpha\varinjlim_\beta \mathrm{cocone}\left( \dgh(Ae, W_\beta) \to\dgh(V_\alpha, W_\beta)\right)
		\\ & \cong \varprojlim_\alpha \mathrm{cocone}\left(\varinjlim_\beta  \dgh(Ae, W_\beta) \to \varinjlim_\beta\dgh(V_\alpha, W_\beta)\right)
		\\ & \cong \varprojlim_\alpha \mathrm{cocone}\left(\varinjlim_\beta  \dgh(Ae, W_\beta) \to \dgh(V_\alpha, \varinjlim_\beta W_\beta)\right)&&
		\\ & \cong \mathrm{cocone}\left(\varinjlim_\beta  \dgh(Ae, W_\beta) \to \varprojlim_\alpha\dgh(V_\alpha, \varinjlim_\beta W_\beta)\right)&&
			\end{align*}
		where the isomorphism in the fourth line follows because each $V_\alpha$ is a bounded complex of finitely generated projectives (cf.\ \ref{indsprmk}), and hence the natural map $\varinjlim_\beta\dgh(V_\alpha, W_\beta) \to \dgh(V_\alpha, \varinjlim_\beta W_\beta)$ is an isomorphism. Let $V_\alpha \into V_{\alpha'}$ be an induced inclusion of induced modules. Then the cokernel is degreewise projective by \ref{ilem}(3). Hence for any dg-$R$-module $W$, the map $\dgh_R(V_{\alpha'},W) \to \dgh_R(V_{\alpha},W)$ is degreewise surjective. In particular, the cofiltered system $U_\alpha\coloneqq \dgh(V_\alpha, \varinjlim_\beta W_\beta)$ of dg vector spaces has surjective transition maps. Because every cofiltered set has a cofinal codirected subset \cite[Expos\'e 1, 8.1.6]{sga4}, to compute the limit of $\{U_\alpha\}_\alpha$ we may as well assume it is a codirected (a.k.a.\ inverse) system. A codirected set, viewed as a partial order, is a Reedy category (see \cite[\S5.2]{hovey}), and it is clear that $\{U_\alpha\}_\alpha$ is a Reedy fibrant diagram of dg vector spaces\footnote{When our codirected set is a copy of $\N$, Reedy fibrancy implies the usual Mittag-Leffler condition.}. Hence we have a quasi-isomorphism $\varprojlim_\alpha U_\alpha \simeq \holim _\alpha U_\alpha$, because one can compute homotopy limits as usual limits along resolved diagrams. But because $\varinjlim_\beta W_\beta$ is acyclic, the obvious morphism of pro-objects $0 \to \{U_\alpha\}_\alpha$ is a levelwise quasi-isomorphism, and it follows that we have quasi-isomorphisms $0\simeq \holim 0 \simeq \holim_\alpha U_\alpha \simeq \varprojlim_\alpha U_\alpha$. Hence $\varprojlim_\alpha\dgh(V_\alpha, \varinjlim_\beta W_\beta)$ is acyclic. Continuing on with the proof, we now have
			\begin{align*}
		 \dge_{D^\mathrm{dg}_\mathrm{sg}(R)}(Ce) & \simeq \varinjlim_\beta  \dgh(Ae, W_\beta) &&
		\\ & \cong  \varinjlim_\beta \mathrm{cone}\left(\dgh(Ae,V_\beta) \to \dgh(Ae,Ae)\right) 
		\\ & \cong \mathrm{cone}\left(\varinjlim_\beta \dgh(Ae,V_\beta) \to \dge(Ae)\right). 
		\end{align*}
		We remark that if $Ae$ is a perfect $R$-module then we can pass the limit appearing in the last line through the derived hom, and we see that $\dge(Ce)$ is acyclic, as expected. Moving on with the proof, fix $\beta$ and consider $\R\hom_R(Ae, V_\beta)$. Since $V_\beta$ is perfect, and $Ae$ has some finitely generated projective resolution, we can write $\R\hom_R(Ae, V_\beta)\simeq V_\beta \otimes_R \R\hom_R(Ae,R)$. Since $M$ is MCM we have a quasi-isomorphism $\R\hom_R(Ae,R)\simeq\hom_R(Ae,R)$, so that $\R\hom_R(Ae, V_\beta)$ is quasi-isomorphic to the tensor product $V_\beta \otimes_R \hom_R(Ae,R)$. The natural isomorphism $eA\otimes_A Ae \to R$ gives an isomorphism $\hom_R(Ae,R)\cong eA$, so we get quasi-isomorphisms $\dgh(Ae,V_\beta)\simeq \R\hom_R(Ae, V_\beta)\simeq V_\beta \otimes_R eA$. So we have $$\varinjlim_\beta \dgh(Ae,V_\beta)\simeq\varinjlim_\beta(  V_\beta \otimes_R eA)\simeq \mathrm{Bar}(Ae)\otimes_R eA\simeq T.$$Hence, we have $\dge(Ce) \simeq \mathrm{cone}(T \to \dge(Ae))$. From the description above, we see that the map $T \to \dge(Ae)$ is exactly the multiplication map $\mu$. More precisely, $xe \otimes ey \in T^0$ is sent to the derived endomorphism that multiplies by $xey \in AeA$ on the left. Putting $B'\coloneqq \tau_{\leq 0} \dge(Ce)$, we hence have a quasi-isomorphism $B'\simeq \mathrm{cone}(T \xrightarrow{\mu} A)$. 
		
		\p By \ref{dqexact}, one has a quasi-isomorphism $B\simeq \mathrm{cone}(T \xrightarrow{\mu} A)$, so that we can already conclude that $B$ and $B'$ are abstractly quasi-isomorphic as complexes. We wish to show that this quasi-isomorphism is induced by the comparison map. To identify this map $\Xi: B \to B'$ we use the explicit description of $B$ given in \ref{drinfeldmodel}. Take a tensor $t=x\otimes r_1 \otimes \cdots \otimes r_i \otimes y\in B$; we remark that we allow $i=0$, in which case the convention is that $t\in A$. The action of $B$ on itself is via concatenating tensors; i.e.\ $$t.\left(w\otimes s_1 \otimes \cdots \otimes s_j \otimes z\right)=x\otimes r_1 \otimes \cdots \otimes r_i \otimes yw\otimes s_1 \otimes \cdots \otimes s_j \otimes z$$(one can check that this also holds when at least one of $i$ or $j$ are zero).
		
		\p We think of the map $\Xi: B\to \dge(Ce)$ as giving an action of $B$ on $Ce$. More precisely, one takes the action of $B$ on itself and sends it through the map $\Sigma$ to obtain an action of $B$ on $Ce$. It is not hard to check this action of $B$ on $Ce$ is via concatenation; more precisely take $c=a\otimes l_1 \otimes \cdots \otimes l_k \otimes b \in Ce \simeq \mathrm{cone}(\mathcal{I} \to Ae)$; then with notation as before we have $$t.c=x\otimes r_1 \otimes \cdots \otimes r_i \otimes ya\otimes l_1 \otimes \cdots \otimes l_k \otimes b$$(which remains true for $i=0$). Because $Ce$ is an ind-object, when we write this we mean that $c$ is an element of some level $(Ce)_\alpha$, and $t.c$ is an element of some other level $(Ce)_\beta$, and we identify $c$ as an element of $(Ce)_\beta$ along the canonical map $(Ce)_\alpha \to (Ce)_\beta$. With the convention that for $k=0$, the element $a\otimes l_1 \otimes \cdots \otimes l_k \otimes b$ is just an element of $Ae$, the above is also true for $k=0$.
		
		\p I claim that across the identification $B'\simeq \mathrm{cone}(T \xrightarrow{\mu} A)$, the action of $B'$ on $Ce$ is precisely the action described above. Showing this claim will prove the theorem, since we have then factored the dga map $\tau_{\leq 0}\Xi$ into a composition of two quasi-isomorphisms $B \to \mathrm{cone}(T \xrightarrow{\mu} A) \to B'$.
		
		\p We saw that we could write $\tau_{\leq 0}B'$ as a cone of the form $\mathrm{cone}\left(\varinjlim_\beta \dgh(Ae,V_\beta) \to A\right)$ where the left hand part acts on $Ce$ by sending $Ae$ into $\mathcal{I}=\{V_\beta\}_\beta$ in the obvious manner, and the right hand part acts on $Ce$ by sending $Ae$ into itself by multiplication on the left. It is clear that across the quasi-isomorphism $B \to B'$, the $A$ summand in the cone acts in the correct manner. So we need to check that the action of $\varinjlim\dgh(Ae,V_\beta)$ on $Ce$ corresponds to the concatenation action of $T$ on $Ce$ provided by $\Xi$.
		
		\p But across the quasi-isomorphism $V_\beta \otimes_R eA \simeq \dgh(Ae, V_\beta)$, an element $v\otimes x$ corresponds to the morphism that sends $y \mapsto v\otimes xy$. Taking limits, we see that across the quasi-isomorphism $T \simeq \varinjlim_\beta \dgh(Ae,V_\beta)$, an element $a\otimes\cdots \otimes x$ corresponds to the morphism that sends $y$ to $a\otimes\cdots \otimes xy$. But this is precisely the concatenation action of $T$ on $Ce$.
	\end{proof}

	\begin{rmk}
		One can think of the above theorem as a chain-level enhancement of \ref{buchcohom}: firstly, the proof above writes $\dge(Ce)$ up to quasi-isomorphism as a cone $\cell A \to \R\enn_R(Ae)$ with $\cell A$ in negative degrees. Hence there is a quasi-isomorphism $\tau_{\geq 1}\R\enn_R(Ae) \to \tau_{\geq 1} \dge(Ce)$. Secondly, for $j<-1$, there is an isomorphism $\underline{\ext}_R^j(Ae,Ae)\cong \tor^R_{-j-1}(Ae,eA)$, which one obtains from \ref{derquotcohom}.
	\end{rmk}
	\begin{rmk}Morally, we ought to have $\R\underline{\enn}_R(M) \simeq \dgh_R(\mathbf{CR}(M),M)$, but the right hand side does not admit an obvious dga structure. Note that $\mathbf{CR}(M)$ is glued together out of a projective resolution $P$ and a projective coresolution $Q^\vee$ of $M$, and hence we ought to have $$\R\underline{\enn}_R(M) \simeq \mathrm{cone}\left( \dgh_R(Q^\vee, M) \to \dgh_R(P,M)\right) \simeq \mathrm{cone}\left(\cell A\to \R\enn_R(M)\right)$$which we do indeed obtain during the course of the proof of \ref{qisolem}.
	\end{rmk}
	
	\begin{rmk}\label{mcmrmk}
		Suppose now that $M$ is no longer assumed to be an MCM module (but is still reflexive). Consider the map $$\psi:Ae\lot_R eA \to Ae\lot_R \R\hom_R(Ae,R)$$ induced by the natural inclusion $\phi:eA\simeq \tau_{\leq 0}\R\hom_R(Ae,R) \to  \R\hom_R(Ae,R)$. Analysing the proof of \ref{qisolem}, it can be shown that $H^i\Xi$ is an isomorphism for a fixed $i\leq 0$ if and only if $H^{i+1}\psi$ is. If $\psi$ induces an isomorphism on $H^j$, then so does $e\psi$, since left multiplication by $e$ is exact (it is tensoring over $A$ with the projective module $eA$). But $e\psi$ is precisely the natural map $\phi$. Conversely if $\phi$ induces an isomorphism on $H^j$ then a spectral sequence argument tells us that $\psi$ induces an isomorphism on $H^j$. So, for a fixed $i\leq 0$, we see that $H^i\Xi$ is an isomorphism if and only if $H^{i+1}\phi$ is. Since $\R\hom_R(Ae,R)$ has cohomology concentrated in nonnegative degrees, and $H^0\phi$ is always an isomorphism, we see that $\tau_{\leq 1}\Xi$ is still a quasi-isomorphism. Moreover $H^0\Xi$ is an isomorphism if and only if $\ext^1_R(Ae,R)\cong \ext^1_R(M,R)$ vanishes. Putting this together, we obtain an exact triangle $$\dq \xrightarrow{\Xi} \tau_{\leq 0}\R\underline{\enn}_R(M) \to \ext^1_R(M,R)[0]\to$$of $A$-bimodules.
		\end{rmk}
	
	\subsection{The kernel of $\Sigma$}
	In the published version of this paper, it was asserted that the kernel of $\Sigma$ is precisely $\per_{\mathrm{fg}}(\dq)$. Unfortunately, this seems not to be the case, as pointed out to the author by Hongxing Chen and Zhengfang Wang. In this section, we make some partial progress towards identifying the actual kernel.
	
	\p Let $A$ be a right noetherian $k$-algebra, let $e\in A$ be an idempotent and let $R\coloneqq eAe$. Let $B\coloneqq\dq$ the derived quotient. Assume also that the singularity category of $R$ is idempotent complete.
	
	\p Recall that to define $\Sigma$, we used as an intermediary step the composition $$\Pi:\frac{\per A}{j_!\per R} \to \Delta_R(A) \xrightarrow{j^*} D_\mathrm{sg}R$$where the first map is induced by the inclusion $\per A \into D^b(A)$ and the second map is right multiplication by $e$. 
	
	\begin{defn}
		If $X \in D^b(A)$, then $\overline{X}$ denotes its image in $\Delta_R(A)$.
		\end{defn}
	As in \ref{kerjlem} and its proof, the kernel of $j^*$ is $\overline{\thick_{D(A)}(\cat{mod}\text{-}A/AeA)}$, and moreover the canonical map $${\thick_{D(A)}(\cat{mod}\text{-}A/AeA)}\to \overline{\thick_{D(A)}(\cat{mod}\text{-}A/AeA)}$$ is an equivalence. Put $\mathcal{P}\coloneqq\per(A) \cap \per_\mathrm{fg}B$. Again, as in the proof of \ref{kerjlem} we see that the projection $\mathcal{P} \to \overline{\mathcal{P}}$ is an equivalence.
	\begin{prop}
		The kernel of $\Pi$ is precisely $\overline{\mathcal{P}}$.
		\end{prop}
	\begin{proof}
		The kernel of $\Pi$ consists of those objects $\overline X$, with $X \in \per A$, such that there exists $X' \in \thick_{D(A)}(\cat{mod}\text{-}A/AeA)$ with an isomorphism $\overline X \cong \overline{X'}$ in $\Delta_R(A)$. Clearly, $\overline{\mathcal{P}}$ is a subcategory of the kernel, because if $X \in \mathcal{P}$ we may just take $X'=X$. To go the other way is more difficult. Suppose that $X$ is in the kernel, and take an appropriate $X'$ as above. Then, \cite[2.1.32]{neemantxt} tells us that we have an  isomorphism $\overline X \cong \overline{X'}$ if and only if:
		\begin{itemize}
			\item There exists $P \in D^b(A)$ and a roof $X \from P \to X'$ such that $\mathrm{cone}(P \to X)\in j_!\per R$.
			\item There exists $Q \to P$ such that $\mathrm{cone}(Q \to X')\in j_!\per R$.
			\item There exists $X' \to Y$ such that $\mathrm{cone}(P \to Y)\in j_!\per R$.
		\end{itemize}
		Recall from \ref{recoll} that we have functorial triangles $$j_!j^!U \to U \to i_*i^*U \to$$ for any $U\in D(A)$. First, apply this to the triangle forming $\mathrm{cone}(Q \to X')$ to conclude that the natural map $i_*i^*Q \to X'$ is an isomorphism. Next, apply $i_*i^*$ to the sequence $Q \to P \to X'$ to conclude that the isomorphism $i_*i^*Q \to X'$ factors through $i_*i^*P$. We observe that this makes $X'$ into a retract of $i_*i^*P$ and hence we have $i_*i^*P\cong X'\oplus Z$. We may moreover choose this isomorphism to make the map $i_*i^*P\to X'$ into the projection.
		
		\p Because $i_*D(B)$ is a thick subcategory of $D(A)$, we must also have $Z\in D(B)$. Because $i_*$ is an embedding we hence have $i^*P \cong X'\oplus Z$. Applying $i_*i^*$ to the map $P \to X$ gives us an isomorphism $i_*i^*P \to i_*i^*X$, and hence an isomorphism $i^*X \cong X'\oplus Z$. The left hand side of this is perfect over $B$, because $i^*$ respects compact objects. Hence the right hand side, and in particular $X'$, is perfect over $B$. So $X'$ is in the intersection of $\per B$ and $\thick_{D(A)}(\cat{mod}\text{-}A/AeA)$,which is exactly $\per_\mathrm{fg}B$.

		\p Similarly, apply the functorial decomposition to the completion of the map $P \to X'$ to an exact triangle. Because $X' \to Z[1]$ is the zero map, we see that $P$ splits as $P\cong X'\oplus Z'$ with $i_*i^*Z'\cong Z$. We hence have a cone $$X' \oplus Z'\to X \to W \to$$ with $W\in j_!\per R$. The rightmost two terms are perfect over $A$, and hence the leftmost term is perfect over $A$. In particular, $X'$ is perfect over $A$.
		\p Hence, $X'$ is an object of $\per(A) \cap \per_\mathrm{fg}B\eqqcolon \mathcal{P}$. Hence $\overline X \cong \overline{X'}$ is an object of $\overline {\mathcal{P}}$, as required.
		\end{proof}
	
	\begin{lem}\label{ksigl}Assume that one of the two following conditions hold.
		\begin{enumerate}
			\item $A$ has finite right global dimension
			\item $A/AeA$ is a finite-dimensional algebra, and all of its simple modules are perfect over $B$.
			\end{enumerate}
		Then we have $\mathcal{P} = \per_\mathrm{fg}B= D_\mathrm{fg}B$ .
		\end{lem}
	\begin{proof}
		In the first case, to show the first equality use that all $A/AeA$-modules are necessarily perfect over $A$ and so $\per_\mathrm{fg}B\subseteq \per A$. To show the second equality, use that $i^*$ respects compact objects and so we have $D_{\text {fg}}B \cong i^*i_*D_{\text {fg}}B\subseteq \per B$. In the second case, argue as in \ref{persrmk}.
		\end{proof}
	
	\begin{prop}\label{kersspan}
		There are inclusions $$\per_\mathrm{fg}B\supseteq i^*\mathcal{P} \subseteq \ker \Sigma$$between subcategories of $\per B$. If either of the two conditions of \ref{ksigl} holds, the leftmost inclusion becomes an equality and hence we have $$\per_\mathrm{fg}B \subseteq \ker \Sigma.$$If $\frac{\per A}{j_!\per R}$ is idempotent complete, then the rightmost inclusion becomes an equality and hence we have $$\per_\mathrm{fg}B \supseteq \ker \Sigma.$$
		\end{prop}
	\begin{proof}
	
	We have an exact sequence $$0 \to j_!\per R \to \per A \xrightarrow{i^*} \per B$$which becomes right exact upon idempotent completion. So $i^*$ descends to a fully faithful map $$\iota:\frac{\per A}{j_!\per R} \to \per B$$ that becomes an equivalence on idempotent completion. Because $\overline {\mathcal{P}}$ is a copy of $\mathcal{P}$, we have $\iota\overline {\mathcal{P}} = i^*\mathcal{P}$. We have $i^*\per_\mathrm{fg}B = \per_\mathrm{fg}B$, because $i^*$ is the identity on $D(B)$. Hence $i^*\mathcal{P} \subseteq \per_\mathrm{fg}B$, and in the situation of \ref{ksigl} this inclusion is an equality.
	
	\p Passing to idempotent completions, observe that $\mathcal{P}\cong \overline{\mathcal{P}}$ is already idempotent complete. It follows that $\mathcal{P}$ is naturally a subcategory of $\left({\frac{\per A}{j_!\per R}}\right)^\omega$. We have $\iota^\omega (\overline{\mathcal{P}}^\omega) = i^*\mathcal{P}$ as subcategories of $\per B$. We have a span $$\per B \xleftarrow{\iota}\frac{\per A}{j_!\per R} \xrightarrow{\Pi} D_\mathrm{sg}R$$that upon idempotent completion becomes a span $$\per B \xleftarrow{\simeq}\left({\frac{\per A}{j_!\per R}}\right)^\omega \xrightarrow{\Pi^\omega} D_\mathrm{sg}R$$where if we invert the first arrow we get precisely the singularity functor $\Sigma$. So we have $\ker \Sigma = \iota^\omega\ker \Pi^\omega $. Certainly $\mathcal{P}\subseteq \ker \Pi^\omega$. So $i^*\mathcal{P}\subseteq \ker \Sigma$. 
	
	\p When $\frac{\per A}{j_!\per R}$ is idempotent complete, $\iota$ is already an equivalence. Hence we do not need to take the idempotent completion and so we have $\ker \Sigma = \iota (\ker \Pi) = \iota \overline{\mathcal{P}} = i^*\mathcal{P}$.
		\end{proof}

	\begin{cor}\label{kerscor}
		Suppose that \begin{enumerate}
			\item Either of the two conditions of \ref{ksigl} holds.
						\item $\frac{\per A}{j_!\per R}$ is idempotent complete.
			\end{enumerate}
		Then the kernel of $\Sigma$ is precisely $\per_{\mathrm{fg}}B$.
		\end{cor}

	\begin{rmk}
		In general, without idempotent completions we have a sequence of functors $$\Sigma\vert_{\im \iota}: \im\iota \xrightarrow{\iota^{-1}}\frac{\per A}{j_!\per R} \to D_\mathrm{sg}R.$$ Clearly $B \in \im \iota$, because $B = i^*A$. The image of $\Sigma\vert_{\im \iota}$ may be smaller than the image of $\Sigma$, because $\im \iota$ need not be thick (it is thick if and only if $\iota$ was surjective in the first place). However we still of course have $Ae \in \im (\Sigma\vert_{\im \iota})$. It's not hard to see that $\ker(\Sigma\vert_{\im \iota}) = i^*\mathcal{P}\subseteq \per_\mathrm{fg}B$. Under the conditions of \ref{ksigl}, we hence have $\ker(\Sigma\vert_{\im \iota}) =  \per_\mathrm{fg}B$.
		\end{rmk}

	\section{Hypersurfaces and periodicity}
	When $R$ is a complete local isolated hypersurface singularity, a famous theorem of Eisenbud \cite{eisenbudper} states that singularity category of $R$ is 2-periodic, in the sense that the syzygy functor $\Omega$ squares to the identity (and in particular $\Omega \cong \Omega^{-1}$). In this section, we show that this periodicity is detected in the derived exceptional locus of a noncommutative partial resolution of $R$. We extend \ref{qisolem} to show that in this situation the comparison map $\Xi:\dq \to \R\underline{\enn}_R(M)$ is a derived localisation at a `periodicity element' $\eta \in H^{-2}(\dq)$. This will give us very tight control over the relationship between $\dq$ and $D_\mathrm{sg}(R)$.
	
	\subsection{Periodicity in the singularity category}
	When $R$ is a commutative complete local hypersurface singularity, it is well-known that the syzygy functor is 2-periodic:
	\begin{thm}[Eisenbud {\cite[6.1(\textit{ii})]{eisenbudper}}]\label{ebudper}
		Let $R$ be a commutative complete local hypersurface singularity over $k$. Then $\Omega^2\cong \id$ as an endofunctor of $\stab R$.
	\end{thm}
	\begin{rmk}
		Eisenbud proves that minimal free resolutions of a MCM module without free summands are 2-periodic. Adding free summands if necessary, one can show that MCM modules always admit 2-periodic resolutions. One can decompose these into short exact sequences to see that the syzygies of $M$ are 2-periodic as claimed.
	\end{rmk}
	
	Recall that $\underline{\ext}^j_R \cong \ext^j_R$ for $j>0$ (\ref{stabextisext}). Combining this with AR duality \ref{arduality} immediately gives us:
	\begin{prop}\label{arduality2}
		Let $R$ be a commutative complete local isolated hypersurface singularity over $k$. Let $M,N$ be MCM $R$-modules. If the Krull dimension of $R$ is even, then one has an isomorphism $$\underline{\hom}_R(M,N) \cong {\ext}_R^1(N,M)^*.$$If the Krull dimension of $R$ is odd, then one has an isomorphism $$\underline{\hom}_R(M,N) \cong {\ext}_R^1(\Omega N,M)^*.$$
	\end{prop}
	We can immediately deduce that in the odd-dimensional case, the stable endomorphism algebra is a symmetric algebra:
	\begin{prop}
		Let $R$ be a commutative complete local isolated hypersurface singularity of odd Krull dimension over $k$. Let $M$ be a MCM $R$-module. Then the stable endomorphism algebra $\Lambda\coloneqq\underline{\enn}_R(M)$ is a symmetric algebra; i.e.\ there is an isomorphism of $\Lambda$-bimodules $\Lambda\cong \Lambda^*$ between $\Lambda$ and its linear dual $\Lambda^*$.
	\end{prop}
	\begin{proof}
		This is essentially \cite[7.1]{bikr}; we follow the explicit proof given in \cite[3.3]{jennyf}. Let $N$ be another MCM $R$-module and put $\Gamma\coloneqq \underline{\enn}_R(N)$. Because $R$ has odd Krull dimension, \ref{arduality2} tells us that we have a functorial isomorphism $$\underline{\hom}_R(M,N) \cong {\ext}_R^1(\Omega N,M)^*$$of $\Lambda$-$\Gamma$-bimodules. But stable Ext agrees with usual Ext in positive degrees by \ref{stabextisext}, so we have functorial isomorphisms $${\ext}_R^1(\Omega N,M) \cong \underline{\ext}_R^1(\Omega N,M)\cong \underline{\ext}_R^0(N,M)\cong \underline{\hom}_R(N,M)$$of $\Gamma$-$\Lambda$-bimodules, because $\Omega$ is the shift. Hence we get a functorial isomorphism $$\underline{\hom}_R(M,N) \cong \underline{\hom}_R(N,M)^*$$of $\Lambda$-$\Gamma$-bimodules. Now put $N=M$.
	\end{proof}

	\begin{defn}\label{rigiddefn}
		Call $M \in \stab R$ \textbf{rigid} if $\ext_R^1(M,M)\cong 0$.
	\end{defn}
	The following is clear:
	\begin{cor}
		Let $R$ be a commutative complete local isolated hypersurface singularity over $k$ of even Krull dimension. If $M$ is a MCM module then it is rigid if and only if it is projective.
	\end{cor}
	We will show that 2-periodicity in the singularity category is detected by the derived stable hom-complexes and the derived stable endomorphism algebras. As a warm-up, we will show that this periodicity appears in the stable Ext-algebras.
	
	\begin{lem}
		Let $R$ be a commutative complete local hypersurface singularity over $k$. Then there are functorial isomorphisms $\underline{\ext}_R^j(M,N)\cong \underline{\ext}_R^{j-2}(M,N)$ for all MCM $R$-modules $M$ and $N$.
	\end{lem}
	\begin{proof}
		There are quasi-isomorphisms $$\R\underline{\hom}_R(M,N)\simeq \R\underline{\hom}_R(M,\Omega^{-2}N)\simeq \R\underline{\hom}_R(M,N)[-2]$$where the first exists by assumption and the second exists since $\Omega^{-1}$ is the shift functor of $\stab R$. Now take cohomology.
	\end{proof}
	\begin{cor}\label{perext}
		Let $R$ be a commutative complete local hypersurface singularity over $k$, and let $M,N$ be MCM $R$-modules. Then for any integers $i,j$ with $i> -j/2$, there are functorial isomorphisms $\underline{\ext}_R^{j}(M,N)\cong \ext_R^{j+2i}(M,N)$. In particular, if $j<0$ then one has an isomorphism $\underline{\ext}_R^{j}(M,N)\cong {\ext}_R^{-j}(M,N)$. 
	\end{cor}
	\begin{proof}
		Periodicity in the stable Ext groups gives isomorphisms $\underline{\ext}_R^{j}(M,N)\cong \underline{\ext}_R^{j+2i}(M,N)$. By assumption, $j+2i>0$ so that $\underline{\ext}_R^{j+2i}(M,N)$ agrees with the usual Ext group. The second assertion follows from taking $i=-j$.
	\end{proof}
	Recall that by definition each MCM $R$-module $M$ comes with a syzygy exact sequence $$0\to \Omega M \to R^a \to M \to 0$$and one in particular has exact sequences of the form $$0\to \Omega^{i+1} M \to R^{a_i}\to \Omega^{i}M \to 0$$for all $i\geq 0$. One can stitch these together into a finite-rank free resolution of $M$. In particular, if $\Omega^2\cong \id$ then one can take $\Omega^{i+2}M=\Omega^iM$, and stitch the syzygy exact sequences together into a $2$-periodic free resolution. The endomorphism algebra of such a resolution detects the 2-periodicity:
	\begin{defn}
		Let $R$ be a commutative complete local hypersurface singularity over $k$, and let $M$ be a MCM $R$-module. Let $\tilde M$ be a $2$-periodic free resolution of $M$. A \textbf{periodicity witness} for $\tilde M$ is a central cocycle $\theta \in \dge^{2}_R( \tilde M)$ whose components $\theta_i:\tilde M ^{i-2} \to \tilde M ^ i$ for $i \leq 0$ are identity maps, up to sign. 
	\end{defn}
	It is clear from the above discussion that periodicity witnesses exist. Because $\dge_R(\tilde M)$ is a model for the derived endomorphism algebra $\R\enn_R(M)$, a periodicity witness hence defines an element of $\ext^2_R(M,M)$. However, note that having a periodicity witness is not a homotopy invariant concept: an element of $\ext^2_R(M,M)$ always lifts to a cocycle in any model for $\R\enn_R(M)$, but need not lift to a central one whose components are identities. We will return to this later. Witnessing elements allow us to explicitly produce a periodic model for the derived stable endomorphism algebra:
	\begin{prop}\label{periodicend}
		Let $R$ be a commutative complete local hypersurface singularity over $k$, and let $M$ be a MCM $R$-module. Let $\tilde M$ be a $2$-periodic free resolution of $M$, with periodicity witness $\theta$. Then there is a quasi-isomorphism of dgas $$\R\underline{\enn}_R(M) \simeq \dge_{R}(\tilde M)[\theta^{-1}] .$$
	\end{prop}
	\begin{proof}We will use \ref{drinfeldlem}. Let $V_n$ be $\tilde M [2n]$, that is, $\tilde M$ shifted $2n$ places to the left. We see that the $V_n$ fit into a direct system with transition maps given by $\theta$. It is not hard to see that $\varinjlim_n V_n$ is acyclic. Projection $\tilde M \to V_n$ defines a map in $\cat{ind}(D^b(R))$ whose cone is clearly ind-perfect, since $V_n$ differs from $\tilde M$ by only finitely many terms. In other words, we have computed $$\R \underline{\enn}_R(M)\simeq \varprojlim_m\varinjlim_n\dgh_{ R}(V_m, V_n)$$Temporarily write $E$ for $\dge_{R}(\tilde M)$, so that $\dgh_{ R}(V_m, V_n)\cong E[2(n-m)]$. Now, the direct limit $\varinjlim_n E[2(n-m)]$ is exactly the colimit of $E[-2m]\xrightarrow{\theta}E[-2m]\xrightarrow{\theta}E[-2m]\xrightarrow{\theta}\cdots$, which is exactly $E[-2m][\theta^{-1}]$. This dga is $2$-periodic, and in particular $E[-2m][\theta^{-1}] \xrightarrow{\theta} E[-2(m+1)][\theta^{-1}]$ is the identity map. Hence $\varprojlim_m E[-2m][\theta^{-1}]$ is just $E[\theta^{-1}]$, as required.
	\end{proof}
	We can state a similar result for the derived stable hom-complexes. Morally, one gets these by periodicising the unstable derived hom-complexes:
	\begin{prop}\label{periodichom}
		Let $R$ be a commutative complete local hypersurface singularity over $k$, and let $M,N$ be MCM $R$-modules. Then the derived stable hom-complex $\R\underline{\hom}_R(M,N)$ admits a $2$-periodic model.
	\end{prop}
	\begin{proof}
		As before, let $\tilde M$ be a periodic resolution for $M$ and write $M_n\coloneqq \tilde M [2n]$. Similarly let $\tilde N$ be a periodic resolution for $N$ and write $N_n\coloneqq \tilde N [2n]$. Then as before one has a quasi-isomorphism $$\R\underline{\hom}_R(M,N)\simeq \varprojlim_m \varinjlim_n E[2n][-2m]$$where we write $E\coloneqq \dgh_R(\tilde M, \tilde N)$, which is a model for $\R\hom_R(M,N)$. The inner colimit $E'\coloneqq \varinjlim_n E[2n]$ is a periodic complex, and the transition maps in the limit $\varprojlim_m E'[-2m]$ all preserve this periodicity, and so the limit is periodic.
	\end{proof}

	\begin{rmk}
		Instead of specifying that $R$ is a commutative complete local hypersurface singularity over $k$, one might want to consider the seemingly more general case when $\Omega^p\cong \id$ for some $p \geq 1$. But if $R$ is a commutative Gorenstein local $k$-algebra satisfying $\Omega^p\cong \id$ for some $p$ then, following the proof of \cite[5.10(4)$\implies$(1)]{crollperiodic}, the $R$-module $k$ is eventually periodic and has bounded Betti numbers. Hence $R$ must be a hypersurface singularity by Gulliksen \cite[Cor. 1]{gulliksen}, and in particular one can take $p=2$.
\end{rmk}

	\subsection{Periodicity in the derived quotient}\label{dqperiod}
Assume in this part that $R$ is a complete local isolated hypersurface singularity and that $M$ is a MCM $R$-module. Let $A\coloneqq \enn_R(R\oplus M)$ be the associated noncommutative partial resolution and put $e\coloneqq \id_R$. Note that by \ref{sgidemref}, the singularity category of $R$ is idempotent complete, so that we are in the situation of Setup \ref{sfsetup}. Because $R$ is a complete local hypersurface singularity, by \ref{ebudper} the shift functor of $\stab R$ is 2-periodic. We show that this periodicity is detected in the derived exceptional locus. The following lemma is useful:
\begin{lem}\label{dqcohom}
	Let $j \in \Z$. Then there are isomorphisms
	
	$$H^j(\dq)\cong \begin{cases} 0 & j>0
	\\ \underline{\enn}(M) & j=0
	\\ \ext_R^{-j}(M,M) & j<0\end{cases}
	$$
\end{lem}
\begin{proof}
	The only assertion that is not clear is the case $j<0$. But in this case, by \ref{perext} we have isomorphisms $\underline{\ext}^j_R(M,M)\cong \ext^{-j}_R(M,M)$. Hence by \ref{qisolem} we have isomorphisms $H^{j}(\dq)\cong \ext^{-j}_R(M,M)$ for all $j<0$. 
\end{proof}
The extra structure given by periodicity allows us to have good control over the relationship between $\dq$ and $\R \underline{\enn}_R(M)$.
\begin{defn}
	If $W$ is a dga and $w\in H(W)$ is a cohomology class, say that $w$ is \textbf{homotopy central} if it is central in the graded algebra $H(W)$. We abuse terminology by referring to cocycles in $W$ as homotopy central.
\end{defn}
Recall from the previous section the existence of an invertible homotopy central cohomology class $\Theta=[\theta]$ in ${\ext}^{2}_R(M,M)$ such that multiplication by $\Theta$ is an isomorphism.

\begin{thm}\label{etaex}Let $\Xi$ be the comparison map.\hfill
	\begin{enumerate}
		\item There is a degree $-2$ homotopy central class $\eta \in H^{-2}(\dq)$ such that $\Xi(\eta)=\Theta^{-1}$.
		\item Multiplication by $\eta$ induces isomorphisms $H^j(\dq) \xrightarrow{\cong} H^{j-2}(\dq)$ for all $j\leq 0$.
		\item The derived localisation of $\dq$ at $\eta$ is quasi-isomorphic to $\R \underline{\enn}_R(M)$.
		\item The comparison map $\Xi: \dq \to \R \underline{\enn}_R(M)$ is the derived localisation map.
	\end{enumerate}
\end{thm}
\begin{proof}By \ref{qisolem}, the comparison map $\Xi$ is a cohomology isomorphism in nonpositive degrees. The first statement is now clear. The element $\eta$ is homotopy central in $\dq$ because $\Theta$ is homotopy central in $\R \underline{\enn}_R(Ae)$. Since $\Xi$ is a dga map, the following diagram commutes for all $j$:
	$$\begin{tikzcd} H^j(\dq) \ar[r,"\eta"]\ar[d,"\Xi"] & H^{j-2}(\dq)\ar[d,"\Xi"] \\ \underline{\ext}_R^{j}(M,M) \ar[r,"\Theta^{-1}"] & \underline{\ext}_R^{j-2}(M,M)\end{tikzcd} $$
	The vertical maps and the lower horizontal map are isomorphisms for $j\leq 0$, and hence the upper horizontal map must be an isomorphism, which is the second statement. Let $B$ be the derived localisation of $\dq$ at $\eta$. Because $\eta$ is homotopy central, the localisation is flat \cite[5.3]{bcl} and so we have $H(B) \cong H(\dq)[\eta^{-1}]$. In particular, for $j\leq 0$, we have $H^j(B)\cong H^j(\dq)$. The map $\Xi$ is clearly $\eta$-inverting, which gives us a factorisation of $\Xi$ through $\Xi':B \to \R \underline{\enn}_R(M)$. Again, the following diagram commutes for all $i,j$ :
	$$\begin{tikzcd} H^j(B) \ar[r,"\eta^i"]\ar[d,"\Xi'"] & H^{j-2i}(B)\ar[d,"\Xi'"] \\ \underline{\ext}_R^{j}(M,M) \ar[r,"\Theta^{-i}"] & \underline{\ext}_R^{j-2i}(M,M)\end{tikzcd} $$The horizontal maps are always isomorphisms. For a fixed $j$, if one takes a sufficiently large $i$, then the right-hand vertical map is an isomorphism. Hence, the left-hand vertical map must be an isomorphism too. But since $j$ was arbitrary, $\Xi'$ must be a quasi-isomorphism, proving the last two statements.
\end{proof} 
\begin{rmk}\label{rigidrmk}
	If $M$ is rigid (see \ref{rigiddefn}) then we have $H(\dq)\cong A/AeA[\eta]$, but in general $\dq$ need not be formal.
\end{rmk}
Left multiplication by $\eta$ is obviously a map $\dq \to \dq$ of right $\dq$-modules. Since $\eta$ is homotopy central, one might expect $\eta$ to be a bimodule map, and in fact this is the case:
\begin{prop}\label{etahoch}
	The element $\eta$ lifts to an element of $ HH^{-2}(\dq)$, the $-2^\text{nd}$ Hochschild cohomology of $\dq$ with coefficients in itself.
\end{prop}
\begin{proof}
	Using \ref{periodicend} and \ref{etaex} gives us a dga $E$, a genuinely central element $\theta^{-1} \in E^{-2}$, and a dga map $\Xi:\dq \to E$ with $\Xi(\eta)=[\theta^{-1}]$. Since $\theta^{-1}$ is central it represents an element of $HH^{-2}(E)$. Because $\Xi$ is the derived localisation map, we have $HH^*(E)\cong HH^*(\dq,E)$ by \ref{hochprop}. Let $C$ be the mapping cone of $\Xi$. Then $C$ is an $\dq$-bimodule, concentrated in positive degrees. We get a long exact sequence in Hochschild cohomology $$\cdots \to HH^n(\dq) \xrightarrow{\Xi} HH^n(\dq,E) \to HH^n(\dq,C)\to \cdots.$$Because $C$ is concentrated in strictly positive degrees and $\dq$ is connective, the cohomology group $HH^n(\dq,C)$ must vanish for $n\leq0$. In particular we get isomorphisms $HH^n(\dq)\cong HH^n(\dq,E)$ for $n\leq0$. Putting this together we have an isomorphism $$HH^{-2}(\dq)\xrightarrow{\Xi} HH^{-2}(\dq,E)\xrightarrow{\Xi} HH^{-2}(E)$$ and it is clear that $\eta$ on the left hand side corresponds to $\theta^{-1}$ on the right.
\end{proof}
\begin{rmk}
	Because $\eta$ is a bimodule morphism, $\mathrm{cone}(\eta)$ is naturally an $\dq$-bimodule. Note that $\mathrm{cone}(\eta)$ is also quasi-isomorphic to the $2$-term dga $\tau_{\geq -1}(\dq)$. This is a quasi-isomorphism of $\dq$-bimodules, because if $Q$ is the standard bimodule resolution of $\dq$ obtained by totalising the bar complex, then the composition $Q \xrightarrow{\eta}\dq \to \tau_{\geq -1}(\dq)$ is zero for degree reasons.
\end{rmk}
\begin{rmk}
	The dga $\dq$ is quasi-isomorphic to the truncation $\tau_{\leq 0}E$, which is a dga over $k[\theta^{-1}]$. Let $H=HH^*_{k[\theta^{-1}]}(\tau_{\leq 0}E)$ be the Hochschild cohomology of the $k[\theta^{-1}]$-dga $\tau_{\leq 0}E$, which is itself a graded $k[\theta^{-1}]$-algebra. One can think of $H$ as a family of algebras over $\A^1$, with general fibre $H[\theta]\cong HH_{k[\theta,\theta^{-1}]}^*(E)$ and special fibre $HH^*(\mathrm{cone}(\eta))$.
\end{rmk}

\begin{prop}\label{uniqueeta}
	Suppose that $A/AeA$ is an Artinian local $k$-algebra. Then $\eta$ is characterised up to multiplication by units in $H(\dq)$ as the only non-nilpotent element in $H^{-2}(\dq)$.
\end{prop}
\begin{proof}
	Let $y\in H^{-2}(\dq)$ be non-nilpotent. Since $\eta: H^0(\dq) \to H^{-2}(\dq)$ is an isomorphism, we must have $y = \eta x$ for some $x \in H^0(\dq)\cong A/AeA$. Since $\eta$ is homotopy central, we have $y^n=\eta^n x^n$ for all $n \in \N$. Since $y$ is non-nilpotent by assumption, $x$ must also be non-nilpotent. Since $A/AeA$ is Artinian local, $x$ must hence be a unit. Note that because $H(\dq)$ is connective, the units of $H(\dq)$ are precisely the units of $A/AeA$.
\end{proof}
\begin{rmk}
	If $A/AeA$ is finite-dimensional over $k$, but not necessarily local, then all that can be said is that $x$ is not an element of the Jacobson radical $J(A/AeA)$. 
\end{rmk}
\begin{prop}\label{uniqueetaqi}
	Let $N$ be another MCM $R$-module and put $B:=\enn_R(R\oplus N)$. Let $e\in B$ denote the idempotent $\id_R$. Suppose that $\dq$ is quasi-isomorphic to $\dqb$. Suppose that $A/AeA\cong B/BeB$ is Artinian local. Then $\R \underline{\enn}_{R}(M)$ and $\R \underline{\enn}_{R}(N)$ are quasi-isomorphic.
\end{prop}
\begin{proof}
	The idea is that the periodicity elements must agree up to units, and this forces the derived localisations to be quasi-isomorphic. Let $\eta_A\in H^{-2}(\dq)$ and $\eta_B\in H^{-2}(\dqb)$ denote the periodicity elements for $\dq$ and $\dqb$ respectively. By assumption, we have a quasi-isomorphism $\dq \to \dqb$; let $\xi \in H^{-2}(\dqb)$ be the image of $\eta_A$ under this quasi-isomorphism. By \ref{uniqueeta}, there is a unit $u\in H^0(\dqb)$ such that $\xi=u.\eta_B$. Because derived localisation is invariant under quasi-isomorphism, we have $\mathbb{L}_{\eta_A}(\dq) \simeq \mathbb{L}_{\xi}(\dqb)$. Observe that if $W$ is a dga, $w \in HW$ any cohomology class, and $v \in HW$ is a unit, then the derived localisations $\mathbb{L}_wW$ and $\mathbb{L}_{vw}W$ are naturally quasi-isomorphic. So we have a chain of quasi-isomorphisms $$\mathbb{L}_{\eta_A}(\dq) \simeq \mathbb{L}_{\xi}(\dqb)\simeq\mathbb{L}_{u\eta_B}(\dqb) \simeq \mathbb{L}_{\eta_B}(\dqb).$$Now the result follows by applying \ref{etaex}(3).
\end{proof}
Since $\R \underline{\enn}_{R}(M)$ is quasi-isomorphic to a dga over $k[\theta,\theta^{-1}]$, and $\R \underline{\enn}_{R}(M)$ is morally obtained from $\dq$ by adjoining $\theta^{-1}$, the following conjecture is a natural one to make:
\begin{conj}\label{perconj}
	If $A/AeA$ is Artinian local then the quasi-isomorphism type of $\dq$ determines the quasi-isomorphism type	of $\R \underline{\enn}_{R}(M)$ as a dga over $k[\theta,\theta^{-1}]$.
\end{conj}

\begin{rmk}
	Note that $\eta$ is a central element of the cohomology algebra $H(\dq)$ but need not lift to a genuinely central cocycle in a model for $\dq$.
\end{rmk}

\begin{rmk}\label{dsgrmk}
	The description of \ref{qisolem} shows that, in this situation, one can compute $\dq$ directly from knowledge of the dg singularity category. This also provides a way to produce an explicit model of $\dq$ where $\eta$ is represented by a genuinely central cocycle: first, stitch together the syzygy exact sequences for $M$ into a 2-periodic resolution $\tilde M \to M$. Let $\theta: \tilde M \to \tilde M$ be the degree 2 map whose components are the identity that witnesses this periodicity. Let $E=\dge_R(\tilde M)$, which is a dga. It is easy to see that $\theta$ is a central cocycle in $E$. Since $\R\underline\enn_R(M)$ is quasi-isomorphic to the dga $E[\theta^{-1}]$, and $\eta$ is identified with $\theta^{-1}$ across this quasi-isomorphism, it follows that $\dq$ is quasi-isomorphic to the dga $\tau_{\leq 0}\left(\dge_R(\tilde M)[\theta^{-1}]\right)$, which is naturally a dga over $k[\eta]=k[\theta^{-1}]$.

\end{rmk}

	\section{A recovery theorem}
	In this section, we prove our main theorem: that the quasi-isomorphism type of the derived exceptional locus of a noncommutative partial resolution of a complete local isolated hypersurface singularity $R$ recovers the isomorphism class of $R$ as a $k$-algebra. The idea is to prove a dg category version of \ref{etaex}, which will allow us to determine the quasi-equivalence class of $D^\mathrm{dg}_\mathrm{sg}(R)$ from the quasi-isomorphism class of $\dq$. We will then apply a recent result of Hua and Keller \cite{huakeller} stating that $D^\mathrm{dg}_\mathrm{sg}(R)$ recovers $R$.
	
	For the remainder of this section, let $R$ be a complete local isolated hypersurface singularity, $M$ a MCM $R$-module,  $A=\enn_R(R\oplus M)$ the associated noncommutative partial resolution, and $e=\id_R \in A$. For brevity, we will often denote the derived exceptional locus by $Q\coloneqq\dq$.

	\subsection{Torsion modules}\label{torsmods}
	By \ref{etaex} there exists a special periodicity element $\eta\in H^{-2}Q$ such that the derived localisation of $Q$ at $\eta$ is the derived stable endomorphism algebra $\R \underline{\enn}_R(M)$. Recall from \ref{colocdga} the construction of the \textbf{colocalisation} $\mathbb{L}^\eta(Q)$ of $Q$, and the fact that an $\eta$-torsion $Q$-module is precisely a module over $\mathbb{L}^\eta(Q)$. 
	\begin{defn}
		Let $\per^bQ$ denote the full triangulated subcategory of $\per Q$ on those modules with bounded cohomology.
	\end{defn}
	Note that $\per^bQ$ is a thick subcategory of the unbounded derived category $D(Q)$.
	\begin{prop}\label{perbisperl}
		The subcategory $\per^bQ$ is exactly $\per\mathbb{L}^\eta(Q)$.
	\end{prop}
	\begin{proof}We show that $\per\mathbb{L}^\eta(Q)\subseteq\per^bQ \subseteq \per\mathbb{L}^\eta(Q)$. Since $\per\mathbb{L}^\eta(Q)=\thick_{D(Q)}(\mathbb{L}^\eta(Q))$, and $ \per^bQ$ is a thick subcategory, to show that $\per\mathbb{L}^\eta(Q)\subseteq\per^bQ$ it is enough to check that $\mathbb{L}^\eta(Q)$ is an element of $\per^bQ$. Put $C\coloneqq\mathrm{cone}(Q \xrightarrow{\eta} Q)$. By construction, the colocalisation $\mathbb{L}^\eta(Q)$ is exactly $\R\enn_Q(C)$. Now, $C$ is clearly a perfect $Q$-module. It is bounded because $\eta$ is an isomorphism on cohomology in sufficiently low degree. As a $Q$-module, we have \begin{align*}
		\R\enn_Q(C) &\simeq \R\hom_Q(\mathrm{cone}(\eta),C)
		\\ &\simeq \mathrm{cocone}\left[\R\hom_Q(Q,C) \xrightarrow{\eta^*} \R\hom_Q(Q,C)\right]
		\\ &\simeq \mathrm{cocone}\left[C \xrightarrow{\eta^*} C\right]
		\end{align*}which is clearly perfect and bounded. Hence $\mathbb{L}^\eta(Q) \in \per^bQ$. To show that we have an inclusion $\per^bQ\subseteq\per\mathbb{L}^\eta(Q)$, we first show that a bounded module is torsion. Let $X$ be any bounded $Q$-module. Then there exists an $i$ such that $X\eta^i\simeq 0$. Choose a $Q$-cofibrant model $L$ for $\mathbb{L}_\eta(Q)$, so that $\mathbb{L}_\eta(X) \simeq X\otimes_Q L$. Then we have $X\otimes_Q L \cong X\otimes_Q\eta^i\eta^{-i} L\cong X\eta^i\otimes_Q\eta^{-i} L\simeq 0$. Now it is enough to show that a perfect $Q$-module which happens to be torsion is in fact a perfect $\mathbb{L}^\eta(Q)$-module. But this is clear: a perfect $Q$-module is exactly a compact $Q$-module, and hence remains compact in the full subcategory of torsion modules.
	\end{proof}
In order to make progress, we will need to know that $\dq$ satisfies some finiteness hypotheses. We take the following definition from Shaul \cite{shaulCM}.
	\begin{defn}
		Say that a dga $W$ is \textbf{noetherian} if 
		\begin{enumerate}
			\item $H^0(W)$ is a noetherian ring.
			\item Each $H^jW$ is finitely generated over $H^0W$.
			\end{enumerate}
	\end{defn}
\begin{prop} The dga $\dq$ is noetherian.
	\end{prop}
\begin{proof}
The stable category of $R$ is hom-finite (e.g.\ \cite{auslanderhomf}). In particular, for all $j\in \Z$ the vector space $\underline{\ext}^j_R(M,M)$ is finite-dimensional. Hence, the graded algebra $H(\dq)$ is finite-dimensional in each degree by \ref{qisolem}, and in particular it is noetherian.
\end{proof}
\begin{rmk}
	By \ref{etaex}(2), the graded algebra $H(Q)$ is finitely generated over $A/AeA$, which is a finite-dimensional algebra. Hence $H(Q)$ is a finitely generated algebra, generated in degrees $0$ through $-2$, with the only degree $-2$ generator being $\eta$. If $M$ is rigid (i.e.\ $\ext^1_R(M,M)\cong 0$) then we have $H(Q)\cong A/AeA[\eta]$. We caution that $Q$ need not be formal.
	\end{rmk}		

	\begin{thm}\label{fintype}
		We have $\per\mathbb{L}^\eta(Q)=\per_\mathrm{fg}(Q)$.
	\end{thm}
	\begin{proof}
		By \ref{perbisperl}, it is enough to show that $\per_\mathrm{fg}(Q)=\per^bQ$. Note that $\per_\mathrm{fg} Q$ is always a subcategory of $\per^bQ$. Since $Q$ is noetherian, we see that for $X\in\per Q$, each $H^jX$ is also finitely generated over $H^0Q$. So a bounded perfect $Q$-module has total cohomology finitely generated over $H^0Q$.
	\end{proof}
\begin{rmk}
	Because $H^0Q$ is finite dimensional, we see that $\per_\mathrm{fg}(Q)$ is the category of perfect $Q$-modules with finite dimensional total cohomology.
	\end{rmk}

\begin{prop}\label{smdg}
Suppose that we are in the situation of \ref{dsgsmooth}. Then 
\begin{enumerate}\item the triangulated categories $\per\mathbb{L}_\eta(Q)$ and $\thick_{D_\mathrm{sg}(R)}(M)$ are triangle equivalent, via the map that sends $\mathbb{L}_\eta(Q)$ to $M$.
	\item The dg categories $\per_{\mathrm{dg}}\mathbb{L}_\eta(Q)$ and $\thick_{D^\mathrm{dg}_\mathrm{sg}(R)}(M)$ are quasi-equivalent, via the map that sends $\mathbb{L}_\eta(Q)$ to $M$.
	\end{enumerate}
	\end{prop}
\begin{proof}
	As in the proof of \ref{dsgsmooth}, we have $\ker\Sigma =\per_\mathrm{fg}(Q)$. 	Hence by \ref{dsgsmooth} and \ref{fintype}, the singularity functor induces a triangle equivalence  $${\bar{\Sigma}}:\frac{\per(Q)}{\per\mathbb{L}^\eta(Q)} \to \thick_{D_\mathrm{sg}(R)}(M)$$which sends $Q$ to $M$. In particular, $\frac{\per(Q)}{\per\mathbb{L}^\eta(Q)}$ is idempotent complete.  By \ref{ntty}, this quotient is precisely $\per\mathbb{L}_\eta(Q)$, and the quotient map sends $Q$ to $\mathbb{L}_\eta(Q)$. For the second statement, by \ref{ntty} we have a quasi-equivalence of dg categories $$
	\per_\mathrm{dg}\mathbb{L}_\eta Q \xrightarrow{\simeq} \frac{\per_\mathrm{dg}(Q)}{\per_\mathrm{dg}\mathbb{L}^\eta(Q)} $$which sends $\mathbb{L}_\eta(Q)$ to $Q$. It is easy to see that the proof of \ref{fintype} gives a quasi-equivalence of dg categories $$\per^\mathrm{dg}_\mathrm{fg}(Q)\simeq \per_\mathrm{dg}\mathbb{L}^\eta(Q)$$ compatible with the inclusion into $\per Q$, and it now follows that the composition
	$$\per_\mathrm{dg}\mathbb{L}_\eta Q \xrightarrow{\simeq} \frac{\per_\mathrm{dg}(Q)}{\per_\mathrm{dg}\mathbb{L}^\eta(Q)}  \xrightarrow{\simeq} \frac{\per_\mathrm{dg}(Q)}{\per^\mathrm{dg}_\mathrm{fg}(Q)}\xrightarrow{\simeq} \thick_{D^\mathrm{dg}_\mathrm{sg}(R)}(M)$$is a quasi-equivalence, where the last map is a quasi-equivalence by \ref{kymapbetterdg}.
	\end{proof}

	\begin{prop}\label{bkprop}Let $\mathcal{T}$ be a pretriangulated dg category and let $X \in \mathcal{T}$ be an object. Then $\thick_\mathcal{T}(X)$ is quasi-equivalent to the dg category $\per_\mathrm{dg} \dge_\mathcal{T}(X)$. 
	\end{prop}
\begin{proof}
	This is an old argument  and essentially goes back to Bondal and Kapranov \cite{bondalkapranov}. As a one-object dg category, $\{X\}$ is equal to the dga $E\coloneqq \dge_\mathcal{T}(X)$. But both $\thick_\mathcal{T}(X)$ and $\per_\mathrm{dg} \dge_\mathcal{T}(X)$ are obtained from $\{X\}$ by closing under shifts and mapping cones (i.e.\ taking the pretriangulated hull). The corresponding quasi-equivalence $\per_\mathrm{dg} \dge_\mathcal{T}(X) \to \thick_\mathcal{T}(X)$ is the obvious one defined by sending the object $E$ to the object $X$.
	\end{proof}
	\begin{thm}\label{recovwk}
When $A/AeA$ is Artinian local, the quasi-isomorphism class of $\dq$ determines the quasi-equivalence class of the dg category $\thick_{D^\mathrm{dg}_\mathrm{sg}(R)}(M)$.
	\end{thm}
	\begin{proof}
		As in the proof of \ref{uniqueetaqi}, the quasi-isomorphism class of $\dq$ determines the quasi-isomorphism class of $\R \underline{\enn}_{R}(M)$, hence the quasi-equivalence class of $\per_{\mathrm{dg}}\R \underline{\enn}_{R}(M)$, and hence the quasi-equivalence class of
		$\thick_{D^\mathrm{dg}_\mathrm{sg}(R)}(M)$ by \ref{bkprop}. 
	\end{proof}

\subsection{A commutative diagram}
We know by \ref{bkprop} that there is an abstract quasi-equivalence of dg categories $$\per_\mathrm{dg}\mathbb{L}_\eta Q \simeq \thick_{D_\mathrm{sg}(R)}(M).$$ Unfortunately it is not clear that this abstract quasi-equivalence can be chosen to be compatible with the maps $\per B \to \per\mathbb{L}_\eta Q $ and $\Sigma:\per B \to\thick_{D_\mathrm{sg}(R)}(M)$. 

\p Under some smoothness assumptions, we do know that such a compatibility holds, by \ref{smdg}. We now a priori have two different ways of getting from $\per_\mathrm{dg}(Q)$ to $\thick_{D_\mathrm{sg}(R)}(M)$: one can either go via the quotient $\frac{\per_\mathrm{dg}(Q)}{\per^\mathrm{dg}_\mathrm{fg}(Q)}$ or via the localisation $\per_\mathrm{dg}\mathbb{L}_\eta Q $. We prove that these are the same. This is not needed for our recovery result, and the uninterested reader can skip this part. For brevity let $E\simeq \R\underline{\enn}_R(M)$ be the derived localisation of ${Q}$ at $\eta$.
\begin{prop}\label{dgcd}Suppose that we are in the situation of \ref{dsgsmooth}. Then there is a commutative diagram in the homotopy category of dg categories $$\begin{tikzcd} \per Q\ar[d,"\pi"]\ar[dr, "\Sigma"]\ar[d] \ar[r, "-\lot_Q E"]& \per(E)\ar[d,"\alpha"] \\ \per Q / \per_\mathrm{fd}(Q) \ar[r, "\bar{\Sigma}", swap] & \thick_{D^\mathrm{dg}_\mathrm{sg}(R)}(M) \end{tikzcd}$$	where $\pi$ is the standard projection functor, $\Sigma$ is the singularity functor, and $\alpha$ and $\bar{\Sigma}$ are quasi-equivalences.
\end{prop}
\begin{proof}
	The bottom left triangle commutes by the definition of $\bar{\Sigma}$, which is an equivalence by \ref{smdg}. We show that the top right triangle commutes. The proof of \ref{etaex} gives a commutative triangle in the homotopy category of dgas $$\begin{tikzcd}
		Q \ar[dr,"\Xi"] \ar[r]& E \ar[d,"\simeq"] \\
		& \R\underline{\enn}_R(M)
	\end{tikzcd}$$where $\Xi$ is the comparison map of \ref{comparisonmap}, $Q \to E$ is the derived localisation at the periodicity element $\eta$, and $E \to \R\underline{\enn}_R(M)$ is a quasi-isomorphism. This gives us a diagram in the homotopy category of dg categories
	$$\begin{tikzcd}
		BQ \ar[dr,"B\Xi"] \ar[r]& BE \ar[d,"\simeq"] \\
		& B\R\underline{\enn}_R(M)
	\end{tikzcd}$$where $BW$ means the dg category with a single object with endomorphism dga $W$. Note that the rightmost map is a quasi-equivalence. Taking perfect modules now gives us a commutative diagram in the homotopy category of dg categories $$\begin{tikzcd}
		\per Q \ar[dr,"F'"] \ar[r,"F"]& \per E \ar[d,"\simeq"] \\
		& \per\left(\R\underline{\enn}_R(M)\right)
	\end{tikzcd}$$where the rightmost map is a quasi-equivalence. It remains to prove that the induced maps $F$ and $F'$ are the correct ones. But if $\mathcal{T}$ and $\mathcal{T}'$ are pretriangulated dg categories where $\mathcal{T}$ is generated by a single object $G$, then any dg functor $\mathcal{T}\to\mathcal{T}'$ is determined by its value on $G$: because objects in $\mathcal{T}$ are generated by $G$ under cones and shifts, their hom complexes are all iterated cones of maps between $\dge(G)$. The same clearly applies for the image of $\mathcal{T}$. So given $G' \in \mathcal{T}'$ and a dg functor of one-object dg categories $G \to G'$, this uniquely extends to a dg functor $\mathcal{T}\to \mathcal{T}'$ by tensoring the hom-complexes in $\mathcal{T}$ with the map $\dge(G)\to \dge(G')$ \cite[Exercice 34]{toendglectures}. In particular it follows that the induced map $F:\per Q \to  \per E$ is the tensor product $-\lot_Q E$. Recall the definition of $\Xi$ from \ref{comparisonmap}: it is the component of the dg functor $\Sigma:\per Q \to \thick_{D^\mathrm{dg}_\mathrm{sg}(R)}(M)$ at the object $Q$. In particular, if one restricts $\Sigma$ to the one object dg category $BQ\subseteq \per Q$, then one gets the dg functor $\Xi$. So $F'\cong\Sigma$.
\end{proof}

\subsection{The main theorem}
We have seen that $\dq$ determines the thick subcategory of the dg singularity category of $R$ generated by $M$ (\ref{recovwk}). A recent theorem of Hua and Keller states that one can recover $R$ from the dg singularity category $D^\mathrm{dg}_\mathrm{sg}(R)$:
\begin{thm}[{\cite[5.7]{huakeller}}]\label{hkthm}
Let $R=k\llbracket x_1,\ldots, x_n \rrbracket / \sigma$ and $R'=k\llbracket x_1,\ldots, x_n \rrbracket / \sigma'$ be two complete local isolated hypersurface singularities. If $D^\mathrm{dg}_\mathrm{sg}(R)$ is quasi-equivalent to $D^\mathrm{dg}_\mathrm{sg}(R')$, then $R\cong R'$ as $k$-algebras.  
\end{thm}
\begin{rmk}
	The proof uses the machinery of singular Hochschild cohomology \cite{kellersing} to identify the zeroth Hochschild homology of $D^\mathrm{dg}_\mathrm{sg}(R)$ with the Tjurina algebra of $R$. This is a $\Z$-graded analogue of Dyckerhoff's theorem \cite{dyck} stating that the Hochschild cohomology of the $\Z/2$-graded dg category of matrix factorisations for $R$ is the Milnor algebra of $R$. The Tjurina algebra of $R$ recovers $R$ by the formal Mather--Yau theorem \cite{matheryau, gpmather}.
	\end{rmk}

	\begin{thm}\label{recov}
		For $i\in\{1,2\}$ let $R_i=k\llbracket x_1,\ldots, x_n \rrbracket / \sigma_i$ be an isolated hypersurface singularity. Let $M_i$ be a non-projective MCM $R_i$-module and let $A_i\coloneqq \enn_{R_i}(R_i\oplus M_i)$ be the associated noncommutative partial resolution with idempotent $e_i=\id_{R_i}$. Suppose that at least one of the $M_i$ is indecomposable. Suppose that there is a quasi-isomorphism $A_1/^\mathbb{L}A_1e_1A_1\simeq A_2/^\mathbb{L}A_2e_2A_2$ between the derived exceptional loci. Then $R_1\cong R_2$ as $k$-algebras.
	\end{thm}
	\begin{proof}
		Because $M_1$ (say) is indecomposable, it follows that the finite-dimensional algebra $A_1/A_1e_1A_1\cong A_2/A_2e_2A_2$ is a local ring. Applying \ref{recovwk} twice now gives us a quasi-equivalence $$\thick_{D^\mathrm{dg}_\mathrm{sg}(R)}(M_1) \simeq \thick_{D^\mathrm{dg}_\mathrm{sg}(R)}(M_2).$$Takahashi's theorem \ref{takthm} (see also \ref{dgtakahashi}) tells us that because $M_i$ is not projective, it generates the stable category, and so this gives a quasi-equivalence $D^\mathrm{dg}_\mathrm{sg}(R_1)\simeq D^\mathrm{dg}_\mathrm{sg}(R_2)$. Now apply \ref{hkthm}.
	\end{proof}

\begin{rmk}
	An MCM module $M$ is not projective if and only if it generates the stable category of $R$. Moreover, this is the case if and only if $\dq\not\simeq 0$. In particular, by \ref{ontolem} if $A$ has finite global dimension then $M$ is not projective.
\end{rmk}

\begin{rmk}\label{tjurrmk}
	In the above situation, one has isomorphisms of algebras $$T_\sigma\cong HH^0(D^\mathrm{dg}_\mathrm{sg}(R))\cong HH^0(\per_{\mathrm{dg}}(\mathbb{L}_\eta(\dq)))\cong HH^0(\mathbb{L}_\eta(\dq)).$$As a vector space, one has $HH^0(\mathbb{L}_\eta(\dq))\cong HH^0(\dq)$ via the proof of \ref{etahoch}. An application of \ref{hochprop} gives an isomorphism $HH^0(\dq)\cong HH^0(A,\dq)$. In particular one can calculate the Tjurina number of the singularity as $\tau_\sigma=\dim_k HH^0(A,\dq)$.
\end{rmk}

\subsection{Threefold flops}
We give a brief sketch of our main application to the homological MMP; for an in-depth discussion, including careful proofs and references, see \cite[Chapter 8]{me}. We hope to present this material more comprehensively in future papers.

\p Let $X \xrightarrow{\pi} \spec R$ be a simple\footnote{i.e.\ the exceptional locus is an irreducible rational curve.} threefold flopping contraction where $X$ has only terminal singularities and $R$ is a complete local ring with an isolated singularity. It follows that $R$ must be an isolated hypersurface singularity (indeed, $R$ is compound du Val by \cite[5.38]{kollarmori}). By results of Van den Bergh \cite{vdb}, Donovan--Wemyss \cite{DWncdf, enhancements}, and Iyama--Wemyss \cite{iyamawemyssfactorial}, there exists a noncommutative partial resolution $A=\enn_R(R\oplus M)$ of $R$ together with a derived equivalence $D^b(A) \to D^b(X)$. Moreover, one may take $A$ to be basic (in the sense of Morita theory) and $M$ to be indecomposable (this relies on the contraction being simple). 

\p The Donovan--Wemyss \textbf{contraction algebra} is defined to be the finite-dimensional algebra $A_\con\coloneqq A/AeA$, and we define the \textbf{derived contraction algebra} to be the derived exceptional locus $\dca\coloneqq \dq$. We regard $\dca$ as an enhancement or categorification of $A_\con$. We note that $M$ is not projective because $A_\con\cong \underline{\enn}_R(M)$ is never the zero ring,  and hence the partial resolution $A$ associated to $\pi$ satisfies the conditions of \ref{recov}. Applying \ref{recov} immediately gives us the following result:

\begin{thm}[Derived Donovan--Wemyss]\label{ddwc}
	Let $X \xrightarrow{\pi} \spec R$ and $X' \xrightarrow{\pi'} \spec R'$ be two simple threefold flopping contractions where $X$ and $X'$ have only terminal singularities and $R$, $R'$ are complete local rings with isolated singularities. Let $\dca$ be the derived contraction algebra of $\pi$ and let $\dca'$ be the derived contraction algebra of $\pi'$. If $\dca$ and $\dca'$ are quasi-isomorphic as dgas, then $R\cong R'$ as $k$-algebras.
	\end{thm}

\begin{ex}\label{singularnope}
In the singular setting, the usual contraction algebra does not classify. To see this, let $R$ be the complete local hypersurface $\frac{k\llbracket x,y,u,v\rrbracket}{uv-x(x^2+y^3)}$, which is an isolated $cA_2$ singularity. By \cite[5.1]{iwreduct} and \cite[4.10]{hmmp}, the $R$-module $M\coloneqq (u,x) \subseteq R$ is MCM, and the associated noncommutative partial resolution $A$ is derived equivalent to a singular minimal model $\pi:X \to \spec R$ constructed from $A$ via quiver GIT. The ring $A$ can be presented as the (completion of the) following quiver algebra\footnote{Our convention for quiver algebras is that $fg$ means `follow arrow $f$, then arrow $g$'.}: $$\begin{tikzcd}
	R \arrow[rr,bend left=15,"x"]\arrow[loop left, "y"] \arrow[rr,bend left=50,"u"]  && M\arrow[loop right, "y"] \arrow[ll,bend left=15,"\text{incl.}"] \arrow[ll,bend left=50,"\frac{v}{x}"]\end{tikzcd}$$(one can obtain this either via a direct calculation or by using \cite[5.33]{iwreduct}). Relabeling, we see that $A$ is isomorphic to the (completion of the) path algebra of the quiver $$
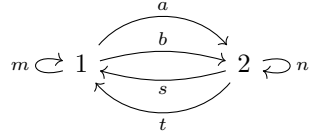
\begin{tikzcd}
	1 \arrow[rr,bend left=15,"b"]\arrow[loop left, "m"] \arrow[rr,bend left=50,"a"]  && 2\arrow[loop right, "n"] \arrow[ll,bend left=15,"s"] \arrow[ll,bend left=50,"t"]\end{tikzcd}$$with the six relations $an=ma$, $bn=mb$, $ns=sm$, $nt=tm$, $at=(bs)^2+m^3$ and $ta=(sb)^2+n^3$. To compute the quotient $A_\con=A/AeA$, one simply has to kill all arrows passing through vertex 1, from which it is clear that $\pi$ has contraction algebra $\frac{k[y]}{y^3}$. However, the (smooth) pagoda flop of width 3, which is a $cA_1$ singularity, is also known to have contraction algebra $\frac{k[y]}{y^3}$ \cite{DWncdf}. So we have exhibited two singular flops with non-isomorphic bases but isomorphic contraction algebras. It follows that $A_\con$ does not classify singular flops. Moreover, if $\pi$ is a minimal model of a terminal threefold, then its derived contraction algebra has cohomology $H(\dq)\cong A_\con[\eta]$ by \cite[9.1.3]{me}. So $H(\dq)$ also does not classify singular flops; one really requires the full dga (or $A_\infty$) structure. We remark that the derived contraction algebra of the pagoda flop is computed in \cite[Chapter 9]{me} using a deformation-theoretic interpretation, and one could also compute the derived contraction algebra of the $cA_2$ flop above using the same method.
\end{ex}

	\phantomsection

\bibliographystyle{alpha}	
\bibliography{thesisbib}

\end{document}